\documentclass[11pt]{amsart}
\usepackage[left=1.5in, right=1.5in]{geometry}
\usepackage[latin1]{inputenc}
\usepackage{amsfonts}
\usepackage[toc,page,title,titletoc,header]{appendix}
\usepackage{graphicx,psfrag,epsfig,multirow,caption,subcaption}
\usepackage{amssymb,amsmath,amscd,amsthm,amssymb,verbatim,hyperref,setspace}
\numberwithin{equation}{section}
\usepackage{mathrsfs,wrapfig}
\usepackage{indentfirst}
\usepackage{extarrows}
\newtheorem{theo}{Theorem}[section]
\newtheorem{lem}{Lemma}[section]
\newtheorem{pro}{Proposition}[section]
\newtheorem{cor}{Corollary}[section]
\newtheorem{defi}{Definition}[section]

\font\uj=cmssbx10 scaled \magstep1

\title[A way to cross double resonance]{A way to cross double resonance}
\author{Chong-Qing Cheng}
\address{Department of mathematics, Nanjing Univerisity, Nanjing 210093, China}
\email{chengcq@nju.edu.cn.}
\begin{document}
\maketitle
\begin{abstract} For typical perturbations of convex integrable Hamiltonian system with three degrees of freedom, a path of diffusion is established to cross strong double resonant point. Together with the uniform hyperbolicity of invariant cylinders got in \cite{C15}, one obtains a transition chain along which one is able to construct diffusion orbits suggested in \cite{A66}.
\end{abstract}
\begin{spacing}{0.5}
\tableofcontents
\end{spacing}
\renewcommand\contentsname{Index}

\section{Introduction and the main result}
\setcounter{equation}{0}
To construct diffusion orbits in nearly integrable Hamiltonian system with three degrees of freedom, one has to handle the problem of double resonance as noticed by Arnold himself \cite{A66}. To be more precise, we study small perturbations of integrable Hamiltonian system
\begin{equation}\label{Eq1}
H(p,q)=h(p)+\epsilon P(p,q),\qquad (p,q)\in\mathbb{R}^3\times\mathbb{T}^3,
\end{equation}
for which we assume that $h$ is convex so that the Hessian matrix $\frac{\partial^2h}{\partial p^2}$ is positive definite and $H,P\in C^r$ with $r\ge 6$. Diffusion orbits are usually constructed along resonant paths. Given an  irreducible integer vector $k'\in\mathbb{Z}^3\backslash\{0\}$, one obtains a resonant path
\begin{equation}\label{resonantpath}
\Gamma_{k'}=\{p\in h^{-1}(\tilde E): \langle k',\partial h(p)\rangle=0\}
\end{equation}
where $\tilde E>\min h$. Along that path, there are points which satisfy another resonant condition, i.e. there exists irreducible integer vector $k''$, independent of $k'$, such that $\langle k'',\partial h(p'')\rangle=0$ holds at that point. We call it strong double resonance if the norm $|k''|$ is not large.

After one step of KAM iteration plus a linear coordinate transformation we obtain a normal form, denoted by $\Phi^*H$. For this normal form the coordinates are still denoted by $(p,q)$. For the Hamiltonian equation produced by $\Phi^*H$, if we introduce another transformation (we call it homogenization)
$$
\tilde G_{\epsilon}=\frac 1{\epsilon}\Phi^*H, \qquad \tilde y=\frac 1{\sqrt{\epsilon}}\Big(p-p''\Big), \qquad \tilde x=q, \qquad s=\sqrt{\epsilon}t,
$$
where $\tilde x=(x,x_3)$, $\tilde y=(y,y_3)$, $x=(x_1,x_2)$, $y=(y_1,y_2)$, the equation still has its own Hamiltonian as the following (cf. \cite{C15}):
\begin{equation}\label{Hamiltonian}
\tilde G_{\epsilon}=\frac{\omega_3}{\sqrt{\epsilon}}y_3+\frac 12\langle\tilde A\tilde y, \tilde y\rangle
-V(x)+\sqrt{\epsilon}\tilde R_{\epsilon}(\tilde x,\tilde y),
\end{equation}
where $\omega_3=\partial_{p_3}h(p'')$, the matrix $\tilde A=\frac{\partial^2h}{\partial p^2}(p'')$ is positive definite, $\tilde R_{\epsilon}(\tilde x,\tilde y)$ is bounded in $C^{r-2}$-topology if $|y|\le K$ with $K$ being independent of $\epsilon$. In the new coordinates $(p,q)\to\Phi(p,q)$ we used to get the normal form $\Phi^*H$, one has $\partial_{p_1}h(p'')=\partial_{p_2}h(p'')=0$. So, it follows from the fact $h({p''})>\min h$ that $\omega_3\ne 0$. The energy level set $\tilde G_{\epsilon}^{-1}(0)$ is uniquely determined by the energy level set $H^{-1}(\tilde E)$. The function $\tilde G_{\epsilon}$ produces a Lagrangian through Legendre transformation
$$
\tilde L_{\epsilon}=\frac 12\langle\tilde A^{-1}(\dot{\tilde x}-\tilde v),(\dot{\tilde x}-\tilde v)\rangle+V(x)+\sqrt{\epsilon}\tilde R'_{\epsilon}(\dot{\tilde x},\tilde x)
$$
where $\tilde v=(0,0,\sqrt{\epsilon}^{-1}\omega_3)$. If $V$ has a unique minimal point, which is assumed at $x=0$ up to a translation of coordinate, and if one ignores the term $\tilde R'_{\epsilon}$, the Mather set of $\tilde L_{\epsilon}$ is a periodic orbit $\tilde x_0(t)=(0,0,x_{3,0}+\frac{\omega_3}{\sqrt{\epsilon}}t)$.

Let $\tilde\alpha$ be the $\alpha$-function of $\tilde L_{\epsilon}$, we are concerned about the classes $\tilde c$ so that $\tilde\alpha(\tilde c)=0$. As the orbit $(\tilde x_0(t),\dot{\tilde x}_0(t))$ supports a minimal measure with non-zero rotation vector for the Lagrangian $\tilde L_{\epsilon}$ if the small term $\sqrt{\epsilon}\tilde R'_{\epsilon}$ is ignored, one has $\min\tilde{\alpha}<0$. Because $\tilde G_{\epsilon}$ is autonomous, $\tilde{\alpha}^{-1}(0)$ is the boundary of a convex set, homeomorphic to a 2-sphere in $\mathbb{R}^3$. Let $e_i\in\mathbb{R}^3$ be the vector whose other two entries are equal to zero except the $i$-th entry which is equal to one, we define
$$
\tilde{\mathbb{F}}_0=\{\tilde c\in H^1(\mathbb{T}^3,\mathbb{R}):\tilde\alpha(\tilde c)=0,\langle\rho(\mu_{\tilde c}),e_i\rangle=0,\forall\ i=1,2\},
$$
where $\rho(\mu_{\tilde c})$ denotes the rotation vector of the minimal measure $\mu_{\tilde c}$ for the class $\tilde c$. $\tilde{\mathbb{F}}_0$ is a flat, by translation, we can assume it is located in the plain $c_3=0$. Restricted in a neighborhood of $\tilde{\mathbb{F}}_0$ the set $\tilde{\alpha}^{-1}(0)$ can be considered as the graph of some Lipschitz function $c\to c_3(c)$ if we use the notation $\tilde c=(c,c_3)$ and $c=(c_1,c_2)$. Now we are ready to state the first main result of this paper.
\begin{theo}\label{mainresult}
For the Hamiltonian $\tilde G_{\epsilon}$ of $($\ref{Hamiltonian}$)$ there is a residual set $\mathfrak{V}\subset C^r(\mathbb{T}^2,\mathbb{R})$ with $r\ge 5$, for each $V\in\mathfrak{V}$ certain number $E_V>0$ exists such that  each contour line $($circle$)$ $(c_3^{-1}(E\sqrt{\epsilon}), E\sqrt{\epsilon})\subset H^1(\mathbb{T}^3,\mathbb{R})$ with $E\in (0,E_V)$ establishes a relation of first cohomology equivalence, namely, any two classes on the same circle are equivalent. These circles make up an annulus-like neighborhood of $\tilde{\mathbb{F}}_0$ in $\tilde{\alpha}^{-1}(0)$.
\end{theo}
The cohomology equivalence for autonomous case was introduced in \cite{LC}. Roughly speaking,
two cohomological classes $c,c'$ are said to be equivalent if there is a continuous path $\Gamma:[0,1]\to H^1(M,\mathbb{R})$ such that $\Gamma(0)=c$, $\Gamma(1)=c'$, $\alpha(\Gamma(s))=\mathrm{constant}$, and for $|s-s'|\ll 1$ $\langle\Gamma(s)-\Gamma(s'),g\rangle=0$ holds for each $g\in H_1(\mathcal{N}(\Gamma(s))|_{\Sigma_{s}},\mathbb{Z})$ where $\Sigma_{s}$ is a suitable section of the configuration manifold. In this case, the Aubry set for $\Gamma(s)$ is connected to that for $\Gamma(s')$ by local minimal orbit, which we call type-$c$. Another mechanism to get local connecting orbit appeared much earlier in \cite{A64}. Since it stays in a small neighborhood of homoclinic orbit, we call it of type-$h$.

It is shown in \cite{C15} that normally hyperbolic invariant cylinders (NHICs) extend to $\sqrt{\epsilon}^{1+d}$-neighborhood of the flat $\tilde{\mathbb{F}}_0$. As $\sqrt{\epsilon}\gg\sqrt{\epsilon}^{1+d}$ if $\epsilon\ll 1$, it follows from Theorem \ref{mainresult} that these NHICs are connected by pathes of cohomology equivalence. Therefore, one obtains candidates of generalized transition chain.

A generalized transition chain is a path $\Gamma:[0,1]\to H^1(M,\mathbb{R})$ such that  $\alpha(\Gamma(s))$ keeps constant and for $|s-s'|\ll 1$ the Aubry set $\tilde{\mathcal{A}}(\Gamma(s))$ is connected to $\tilde{\mathcal{A}}(\Gamma(s'))$ by local minimal orbit either of type-$c$ or of type-$h$. Usually, a chain exists along resonant channel. Recall the resonant relation $k'$ in (\ref{resonantpath}), which determines a resonant path in 
\begin{equation}\label{resonantpathinc}
\tilde{\mathbb{W}}_{k',\tilde E}=\Big\{\tilde c\in H^1(\mathbb{T}^3,\mathbb{R}):\alpha_H(\tilde c)=\tilde E,\langle\rho(\mu_{\tilde c}),k'\rangle=0\Big\},
\end{equation}
where the $\alpha$-function and the $\beta$-function are determined by $H$.
If $\epsilon=0$, $\tilde{\mathbb{W}}_{k',\tilde E}=\Gamma_{k'}$ if we regard them as sets in $\mathbb{R}^3$. The relation $p=\tilde c$ holds along each $\tilde c$-minimal orbit. For small $\epsilon>0$, $\tilde{\mathbb{W}}_{k',\tilde E}\subset\alpha^{-1}(\tilde E)$ becomes a narrow channel, the action variable $p$ along any orbit in $\tilde{\mathcal{A}}(\tilde c)$ remains in a small neighborhood of $p=\tilde c$. A diffusion orbit is obtained if it connects $\tilde{\mathcal{A}}(\tilde c)$ to $\tilde{\mathcal{A}}(\tilde c')$, the orbit intersects a small neighborhood of the torus $p=\tilde c$ and of the torus $p=\tilde c'$.

Let $B_{\tilde E}\subset\mathbb{R}^3$ be a ball centered at the origin such that $H^{-1}(\tilde E)\subset B_{\tilde E}$.

\begin{theo}\label{mainresult2}
For the Hamiltonian $H$ of $($\ref{Eq1}$)$ and a resonant path in $\tilde{\mathbb{W}}_{k',\tilde E}$ of $($\ref{resonantpathinc}$)$, there exists a cusp-residual set $\mathfrak{P}\subset C^r(B\times\mathbb{T}^3)$ or $\mathfrak{P}\subset C^r(\mathbb{T}^3)$ with $r\ge 6$, such that for each $\epsilon P\in\mathfrak{P}$, the Hamiltonian flow $\Phi_H^t$ admits a generalized transition chain along the path in $\tilde{\mathbb{W}}_{k',\tilde E}$.
\end{theo}
Although it appears complicated, using variational method to construct diffusion orbits along a generalized transition chain has become somehow routine work, similar to what we have done in \cite{CY1,CY2,LC}. We shall do it in another paper.

Since the announcement of Mather, it has been widely known as a difficult problem how to cross double resonance  (see \cite{M03}). Along the path suggested by Mather, one tries to follow a cylinder with hole such that the method for {\it a priori} unstable system can be applied (see \cite{KZ,Mar}). In this case, one is involved into the dynamics in the zero energy level set. Although it appeared at first view a trouble that the hyperbolic fixed point breaks invariant cylinder, the intersection of stable and unstable manifold of the fixed point produces a ``stochastic" layer which connects many pieces of cylinder. It allows us to design a completely different way to cross double resonance, certain orbit moves in the ``stochastic" layer connecting one cylinder to another. As one does not touch the zero energy level set along the way, the new method has its application in the problem to cross multiple resonance in {\it a priori} stable system with arbitrary degrees of freedom \cite{CX}. In the way available to handle systems with more than three degrees of freedom, certain singularities merges at zero energy level set during the process of reduction of order. It prevents one from touching the zero energy level.

The proof of Theorem \ref{mainresult} relies on well understanding of the dynamics around the fixed point of the averaged system, we shall do it in Section 2. The proof is completed in Section 3. Together with the result obtained in \cite{C15}, we obtain Theorem \ref{mainresult2} in Section 4.

\section{The dynamics around fixed point}
\setcounter{equation}{0}

Restricted on the energy level $\tilde G_{\epsilon}^{-1}(0)$ (corresponding to the energy level $H^{-1}(E)$), around a double resonant point $p''$, we get an equivalent $2\pi\frac{\sqrt{\epsilon}}{\omega_3}$-periodic Hamiltonian with two degrees of freedom
\begin{equation}\label{mainsystems}
G_{\epsilon}=\frac 12\langle Ay,y\rangle -V(x)+\sqrt{\epsilon}R_{\epsilon}(x,y,\tau),\qquad (x,y)\in\mathbb{T}^2\times\mathbb{R}^2
\end{equation}
where $\tau=\frac{\sqrt{\epsilon}}{\omega_3}x_3$, $G_{\epsilon}$ solves the equation $\tilde G_{\epsilon}(x,\frac{\omega_3}{\sqrt{\epsilon}}\tau,y,-\frac{\sqrt{\epsilon}}{\omega_3} G_{\epsilon})=0$, $\omega_3=\partial_{p_3}h(p'')\ne 0$ and the matrix $A$ is positive definite. If we denote the Hamiltonian with two degrees of freedom by
\begin{equation}\label{averagesystem}
G=\frac 12\langle Ay,y\rangle -V(x),
\end{equation}
then $G_{\epsilon}=G+\sqrt{\epsilon}R_{\epsilon}$ is a time-periodic small perturbation of $G$. For the Hamiltonian $G$, the minimal point of $V$, denoted by $x_0$, determines a constant solution $(x,y)=(x_0,0)$. This point supports a minimal measure of the Lagrangian $L$ which is obtained from $G$ by Legendre transformation
$$
L=\frac 12\langle A^{-1}\dot x,\dot x\rangle+V(x).
$$
The rotation vector of this minimal measure is zero. Let $\alpha_0$ denote the $\alpha$-function for $L$, it reaches its minimum at a set including the zero class $0\in H^1(\mathbb{T}^2,\mathbb{R})$. We denote this set by
$$
\mathbb{F}_0=\{c\in H^1(\mathbb{T}^2,\mathbb{R}):\alpha_0(c)=\min\alpha_0\}.
$$
This set is usually a full dimensional flat of the $\alpha$-function.

\subsection{Flat of the $\alpha$-function}
By the definition, a subset is called a flat of certain $\alpha$-function if, restricted on this set, the $\alpha$-function is affine, and no longer affine on any set properly containing the flat. As $\alpha$-function is convex with super-linear growth, each flat is a convex and bounded set. Given an $n$-dimensional flat $\mathbb{F}$, a subset in $\partial \mathbb{F}$ is called an edge if it is contained in a $(n-1)$-dimensional hyperplane. Since each flat is convex, each edge is also convex.

\begin{theo}\label{flatthm1}
Given a Tonelli Lagrangian $L\in C^2(T\mathbb{T}^n,\mathbb{R})$ and a first cohomology class $c_0\in H^1(\mathbb{T}^n,\mathbb{R})$, if the minimal measure is supported on a hyperbolic fixed point, then there exists an $n$-dimensional flat $\mathbb{F}_0\subset H^1(\mathbb{T}^n,\mathbb{R})$ such that for each $c\in\mathbb{F}_0$, the $c$-minimal measure is supported on this point.
\end{theo}
\begin{proof}
By translation one can assume that the fixed point is at $(x,\dot x)=(0,0)$, by adding a closed 1-form and a constant to the Lagrangian, one can assume $c_0=0$ and $L(0,0)=0$.

To each closed curve $\xi$: $[-T,T]\to\mathbb{T}^n$ with $\xi(-T)=\xi(T)$ a first homology class $[\xi]=g\in H(\mathbb{T}^n,\mathbb{Z})$ is associated. We consider the quantity
$$
A(g)=\liminf_{T\to\infty}\inf_{\stackrel {\xi(-T)=\xi(T)}{\scriptscriptstyle [\xi]=g}} \int_{-T}^TL(d\xi(t))dt.
$$
By the condition assumed on $L$, one has that $A(g)\ge0$ for any $g\ne 0$. There exist at least $n+1$ irreducible classes $g_i\in H(\mathbb{T}^n,\mathbb{Z})$ and $n+1$ minimal homoclinic orbits $d\gamma_i$ such that $A([g_i])=A(\gamma_i)$ \cite{Be}. Clearly, $H_1(\mathbb{T}^n,\mathbb{Z})$ can be generated by
the homology classes of all minimal homoclinic curves over $\mathbb{Z}_+$.

Let us abuse the notation $g$ to denote homology class $g\in H_1(\mathbb{T}^n,\mathbb{Z})$ or to denote a point $g\in\mathbb{Z}^n$. For each curve $\bar\gamma_T$: $[-T,T]\to\mathbb{R}^n$ with $\bar\gamma_T(-T)=0$ and $\bar\gamma_T(T)=g$, one has
$$
A(g)=\liminf_{T\to\infty}\inf_{\stackrel {\bar\xi(-T)=0}{\scriptscriptstyle \bar\xi(T)=g}} \int_{-T}^TL(\dot{\bar\xi}(t),\bar\xi(t))dt.
$$
We claim that
\begin{equation}\label{flateq1}
A(g)+A(-g)>0.
\end{equation}
To prove this claim, let us recall the definition of $h^{\infty}$ introduced in \cite{M93}
$$
h^{\infty}(\bar x,\bar x')=\liminf_{T\to\infty}\inf_{\stackrel {\bar\xi(0)=\bar x}{\scriptscriptstyle \bar\xi(T)=\bar x'}} \int_{0}^TL(\dot{\bar\xi}(t),\bar\xi(t))dt.
$$
but here the configuration coordinate are in the universal covering space, $x,x'\in\mathbb{R}^n$, not on the torus.
It is a backward weak KAM if we think it as the function of $x'$ and is a forward weak KAM of $x$. By the definition, we have
$$
A(g)+A(-g)=h^{\infty}(0,g)+h^{\infty}(g,0).
$$

The quantity $A(g)$ is achieved may not by a curve connecting the origin to $g\in\mathbb{Z}^n$, but by the conjunction of several curves  $\bar\gamma_1\ast\bar\gamma_2\ast\cdots\ast\bar\gamma_m$. Let $\bar\gamma_i$: $\mathbb{R}\to\mathbb{R}^n$ denote a curve ($i=1,\cdots,m$), the conjunction implies that $\bar\gamma_i(-\infty)=\bar\gamma_{i-1}(\infty)$. Let $\gamma_i=\pi_{\infty}\bar\gamma_i$, where $\pi_{\infty}:\mathbb{R}^n\to\mathbb{T}^n$ denotes the standard projection, then, $\gamma_1,\cdots,\gamma_m$ are minimal homoclinic curves so that $g=\sum_{i=1}^m[\gamma_i]$. Let $g_i=\sum_{j=1}^i[\gamma_j]$, both functions $h^{\infty}(0,x)$ and $h^{\infty}(x,g_i)$ are differentiable along the curve $\bar\gamma_i$, the derivative $\partial_2h^{\infty}(0,x)=\partial_1h^{\infty}(x,g_i)$ determines the speed of $\dot{\bar\gamma}_i(t)$ if $x=\bar\gamma_i(t)$. So, some constant $C_i$ exists such that
\begin{equation}\label{flateq2}
h^{\infty}(0,\bar\gamma_i(t))=h^{\infty}(\bar\gamma_i(t),g_i)+C_i,\qquad \forall\ t,
\end{equation}
Let $t\to\infty$, one sees that $C_i=h^{\infty}(0,g_i)$.

Clearly, both functions $h^{\infty}(g_i,x)$ and $h^{\infty}(x,g_i)$ are $L$-dominate function. Restricted on a small ball $B_{\delta}(g_i)$ centered at $g_i$, the projection of these functions $h^{\infty}(\pi_{\infty}g_i,\pi_{\infty}x)$ and $-h^{\infty}(\pi_{\infty}x,\pi_{\infty}g_i)$ are conjugate pair of weak KAM. As $h^{\infty}(g_i,g_i)=0$, one has (see Theorem 5.1.2 in \cite{Fa})
\begin{equation}\label{flateq2.1}
h^{\infty}(x,g_i)\le h^{\infty}(g_i,0)-h^{\infty}(x,0), h^{\infty}(0,x)-h^{\infty}(0,g_i)\le h^{\infty}(g_i,x).
\end{equation}
If
$$
A(g)+A(-g)=h^{\infty}(0,g)+h^{\infty}(g,0)=0,
$$
substituting $h^{\infty}(x,g_i)$ in (\ref{flateq2.1}) by the expression in (\ref{flateq2}) one sees that
$$
h^{\infty}(0,\bar\gamma_m(t))+h^{\infty}(\bar\gamma_m(t),0)=0, \qquad \forall\ t\in\mathbb{R}.
$$
Let $t\to -\infty$ one obtains that $h^{\infty}(0,g_{m-1})+h^{\infty}(g_{m-1},0)=0$. Repeating this procedure, one obtains that
$$
h^{\infty}(0,\bar\gamma_1(t))+h^{\infty}(\bar\gamma_1(t),0)=0, \qquad \forall\ t\in\mathbb{R}.
$$
Because the point $\{(x,\dot x)=0\}$ is hyperbolic, certain large number $t_0>0$ and suitably small number $\delta>0$ exist such that $\bar\gamma(t)\in B_{\delta}(0)$ if $t<-t_0$ and
$$
h^{\infty}(0,\bar\gamma_1(t))+h^{\infty}(\bar\gamma_1(t),0)>0, \qquad \forall\ t<-t_0,
$$
as $\{0\}$ is the unique minimal point of the barrier function in $B_{\delta}(0)$. This contradiction proves the formula (\ref{flateq1}).

Let
$$
\mathbb{G}_0=\{g\in H_1(\mathbb{T}^n,\mathbb{Z}):\exists\ \gamma: \mathbb{R}\to\mathbb{T}^n\ s.t. \ [\gamma]=g,\ A(\gamma)=0\}.
$$
$\mathbb{G}_0$ is said to generate a rational direction $g\in\mathbb{Z}^n$ over $\mathbb{Z}_+$ if there exist $k,k_i\in\mathbb{Z}_+$ and $g_i\in \mathbb{G}_0$ such that
$$
kg=\sum k_ig_i.
$$
It is an immediate consequence of the formula (\ref{flateq1}) that once $\mathbb{G}_0$ generates a rational direction $g\in\mathbb{Z}^n$ over $\mathbb{Z}_+$, then it can not generate the direction $-g$ over $\mathbb{Z}_+$. Therefore, the set
$$
\text{\rm span}_{\mathbb{R}_+}\mathbb{G}_0=\{\Sigma a_ig_i:\ g_i\in\mathbb{G}_0,\ a_i\ge 0\}
$$
is a cone properly restricted in half space. Thus, there exists an $n$-dimensional cone $\mathbb{C}_0$ such that
$$
\langle c,g\rangle >0,\qquad \forall\ c\in\mathbb{C}_0,\ g\in \text{\rm span}_{\mathbb{R}_+}\mathbb{G}_0.
$$

Since the minimal measure for zero cohomology class is supported on the fixed point, $\tilde{\mathcal{N}}(0)$ is composed of those minimal homoclinic orbits along which the action equals zero. According to the upper semi-continuity of Ma\~n\'e set in cohomology class, any minimal measure $\mu_c$ is supported by a set lying in a small neighborhood of these homoclinic orbits if $|c|$ is very small. Consequently. we have $\rho(\mu_c)\in\text{\rm span}_{\mathbb{R}_+}\mathbb{G}_0$, where $\rho(\mu_c)$ denotes the rotation vector of $\mu_c$.

Let us consider a cohomology class $c$ such that $-c\in\mathbb{C}_0$ and $|c|\ll 1$. We claim that the $c$-minimal measure is also supported on the fixed point. Indeed, if it is not true, we would have positive average action of $L$: $ A(\mu_c)>0$, since the minimal measure for zero class is assumed unique and supported on the fixed point. By the choice of $c$ one has that $\langle c,\rho(\mu_c)\rangle<0$. Thus, one obtains
$$
A_c(\mu_c)=A(\mu_c)-\langle c,\rho(\mu_c)\rangle>0=A_c(\mu),
$$
it deduces absurdity. For this class $c$, the action of the Lagrangian $L_c=L-\langle c,\dot x\rangle$ along any minimal homoclinic curve $\gamma$ is positive,
$$
A(\gamma)-\langle c,[\gamma]\rangle >0,
$$
namely, the Aubry set for this class is also a singleton. Consequently, $\mu_{c'}$ is also supported on this point if $c'$ is sufficiently close to $c$. This verifies the existence of $n$-dimensional flat.
\end{proof}

\subsection{Around the flat of minimum}
In this section we restrict ourselves to the special case that the system has two degrees of freedom. The task of this section is to study the structure of the Mather sets as well as of the Ma\~n\'e sets in a neighborhood of the fixed point of the Hamiltonian defined in (\ref{averagesystem}), which we rewrite here
$$
G=\frac 12\langle Ay,y\rangle -V(x).
$$
If $x_0$ is a non-degenerate minimal point of $V$, the flat $\mathbb{F}_0$ is a 2-dimensional disk, the matrix
$$
\left (\begin{matrix}0 & A\\
\partial^2_xV & 0
\end{matrix}\right )
$$
has 4 real eigenvalues $\pm\lambda_1,\pm\lambda_2$. Up to translation of coordinates, it is generic that

({\bf H1}): {\it The potential $V$ attains its minimum at $x=0$ only, the Hessian matrix of $V$ at $x=0$ is positive definite. All eigenvalues are different: $-\lambda_2<-\lambda_1<0<\lambda_1<\lambda_2$}.

This hypothesis leads to some hyperbolicity of minimal homoclinic orbits. For each $c\in\mathrm{int}\mathbb{F}_0$, the Aubry set is a singleton which is fixed point.  Let us consider the case: for $c\in\partial\mathbb{F}_0$ the Aubry set $\mathcal{A}(c)=\cup_{t\in\mathbb{R}}\zeta(t)$, where $\zeta$: $\mathbb{R}\to M$ is a minimal homoclinic curve. By the assumption {\bf H1}, the fixed point $z=(x,y)=0$ has its locally stable manifold $W^+$ as well as the locally unstable manifold $W^-$. They intersect each other transversally at the origin. As each homoclinic orbit entirely stays in the stable as well as in the unstable manifolds, along such orbit their intersection can not be transversal in the standard definition, but transversal module the curve:
$$
T_xW^-\oplus T_xW^+=T_xH^{-1}(E)
$$
holds for $x$ is on minimal homoclinic curves. Without danger of confusion, we call the intersection transversal also.

If we denote by $\Lambda^+_i=(\Lambda_{xi},\Lambda_{yi})$ the eigenvector corresponding to the eigenvalue $\lambda_i$, where $\Lambda_{xi}$ and $\Lambda_{yi}$ are for the $x$- and $y$-coordinate respectively, then the eigenvector for $-\lambda_i$ will be  $\Lambda^-_i=(\Lambda_{xi},-\Lambda_{yi})$. it is also a generic condition that

({\bf H2}): {\it with each $g\in H_1(\mathbb{T}^2,\mathbb{Z})$, there is at most one minimal orbit associated, the stable manifold intersects the unstable manifold transversally along each minimal homoclinic orbit. Each minimal homoclinic orbit approaches to the fixed point along the direction $\Lambda_1$: $\dot\gamma(t)/\|\dot\gamma(t)\| \to\Lambda_{x1}$ as $t\to\pm\infty$.}

As the next step, let us study the shape of the disk $\mathbb{F}_0$. Clearly, $0\in\mathrm{int}\mathbb{F}_0$. We divide the boundary of $\mathbb{F}_0$ into two parts $\partial\mathbb{F}_0=\partial^*\mathbb{F}_0\cup (\partial\mathbb{F}_0\backslash \partial^*\mathbb{F}_0)$, where
\begin{equation*}
\partial^*\mathbb{F}_0=\{c\in\partial\mathbb{F}_0: \ \mathcal{M}(c)\backslash\{x=0\}\neq\varnothing \},
\end{equation*}
We use $\mu_c$ to denote the minimal measure which is not supported on the fixed point. As the configuration space is a $2$-torus, all minimal measures, except the one supported on the fixed point, share the same rotation vector, denoted by $\rho(\mu_c)$.

The set $\partial^*\mathbb{F}_0$ may be non-empty. Here is an example:
$$
L=\frac 12\dot x_1^2+\frac {\lambda^2}2\dot x_2^2+V(x)
$$
where $|\lambda|\ne 1$, the potential satisfies the following conditions: $x=0$ is the minimal point of $V$ only; there exist two numbers $d>d'>0$ such that for any closed curve $\gamma$: $[0,1]\to\mathbb{T}^2$ passing through the origin with $[\gamma]\ne 0$ one has
$$
\int_0^1V(\gamma(s))ds\ge d;
$$
$V=d'+(x_2-a)^2$ when it is restricted a neighborhood of circle $x_2=a$ with $a\ne 0$ mod 1. In this case, $\partial\mathbb{F}_0\cap\{c_2=0\}=\{c_1=\pm\sqrt{2d'}\}$. Indeed,
$$
L\pm c_1\dot x_1=\frac 12(\dot x_1\pm c_1)^2+\frac{\lambda^2}2\dot x_2^2+V(x)-\frac 12c_1^2,
$$
the Mather set for $c=(\pm\sqrt{2d'},0)$ consists of the point $x=0$ and the periodic curve $x(t)=(x_{1,0}\mp\sqrt{2d'}t,a)$.

Clearly, the set $\partial^*\mathbb{F}_0$ is closed with respect to $\mathbb{F}_0$. If it is non-empty, the existence of infinitely many $\bar M$-minimal homoclinic orbits has been proved in \cite{Zhe,Zm2}. These orbits are associated with different homological classes. If $\partial^*\mathbb{F}_0=\varnothing$, there are at least three minimal homoclinic orbits to the fixed point.

The existence of homoclinic orbit to some Aubry set is closely related to the existence of the flat of the $\alpha$-function.
\begin{lem}\label{flatlem1}
Given $c,c'\in\mathbb{F}$, let $c_{\lambda}=\lambda c+(1-\lambda)c'$. Then
$$
\tilde{\mathcal{A}}(c)\cap\tilde{\mathcal{A}}(c')=\tilde{\mathcal{A}}(c_{\lambda}), \qquad \forall\ \lambda\in (0,1).
$$
\end{lem}
\begin{proof}
Using argument in \cite{Mas}, for any curve $\gamma$: $\mathbb{R}\to M$, we have
$$
[A_{c_{\lambda}}(\gamma|_{I})]=\lambda [A_c(\gamma|_{I})]+(1-\lambda)[A_{c'}(\gamma|_{I})],\qquad \forall\
I\subset\mathbb{R}.
$$
As both $\lambda>0$ and $1-\lambda>0$, one has that $[A_c(\gamma)]=[A_{c'}(\gamma)]=0$ if $[A_{c_{\lambda}}(\gamma)]=0$.
\end{proof}
\begin{lem}\label{flatlem2}
Let $\mathbb{F}_0$ be a 2-dimensional flat, the Mather set is a singleton for each class in the interior of $\mathbb{F}_0$, let $\mathbb{E}_i$ be an edge of $\mathbb{F}_0$, then
$$
\mathcal{A}(c')\supsetneq\mathcal{A}(c)
$$
holds for $c'\in\partial\mathbb{F}_0$ $(\partial\mathbb{E}_i)$ and $c\in\mathrm{int}\mathbb{F}$ $(\mathrm{int} \mathbb{E}_i)$ respectively.
\end{lem}
\begin{proof}
Since the Mather set is a singleton for each $c\in\text{\rm int}\mathbb{F}_0$, each orbit in the Aubry set is either the fixed point itself, or a homoclinic orbit to the point with null first homology. Denote by $[\gamma]$ the first homology class of the homoclinic curve $\gamma$ in the Aubry set, then
$$
\int_{-\infty}^{\infty}L(d\gamma(t))dt-\langle c,[\gamma]\rangle=0
$$
holds for all $c\in\text{\rm int}\mathbb{F}_0$. It follows that $\langle c-c',[\gamma]\rangle=0$ for $c,c'\in\text{\rm int}\mathbb{F}_0$. As $\mathbb{F}_0$ shares the same dimension of the configuration space, $[\gamma]=0$. For classical mechanical system of (\ref{averagesystem}), the Aubry set has to be the same as the Mather set for $c\in\text{\rm int}\mathbb{F}_0$. For a class $c'\in\partial\mathbb{F}_0\backslash \partial^*\mathbb{F}_0$, the Aubry set $\mathcal{A}(c')$ contains at least one minimal homoclinic curve with non-zero first homology. Otherwise, for a class $c'\notin\mathbb{F}_0$ very close to $c$, the homology of the Ma\~n\'e set is trivial, the same as that for $c$. It is guaranteed by the upper semi-continuity of Ma\~n\'e set in cohomology class. It follows that $\langle c,\rho(\mu_c)\rangle =\langle c',\rho(\mu_c')\rangle=0$ and
$$
-\alpha(c')=A(\mu_{c'})-\langle c',\rho(\mu_c')\rangle\ge A(\mu_c)=-\alpha(c).
$$
However, as $c'\notin\mathbb{F}_0$, one has $\alpha(c')>\alpha(c)$. The contradiction verifies our claim. If $c'\in\partial^*\mathbb{F}_0$, the certain $c'$-minimal measure $\mu_{c'}$ exists with $\rho(\mu_{c'})\ne 0$. In both cases, $\mathcal{A}(c')\supsetneq\mathcal{A}(c)$ if $c\in\text{\rm int}\mathbb{F}_0$.

Let $\mathbb{E}_i$ be an edge. For $c\in\text{\rm int}\mathbb{E}_i$, the Aubry set contains one or more homoclinic curves, all of them share the same homology class, denoted by $g(\mathbb{E}_i)$ which is of course non-zero. If $\mathcal{M}(c)$ contains other curves, these curves also share the same rotation vector as $\langle c-c',g(\mathbb{E}_i)\rangle=0$ holds for $c,c'\in\text{\rm int}\mathbb{E}_i$.

Let $c'\in\partial\mathbb{E}_i$ and $c\in\text{\rm int}\mathbb{E}_i$, one chooses $c^*\in\partial\mathbb{F}_0 \backslash\mathbb{E}_i$ arbitrarily close to $c'$. As the straight line connecting $c$ to $c^*$ passes through the interior of $\mathcal{F}_0$, we obtain from Lemma \ref{flatlem1} that $\mathcal{A}(c)\cap \mathcal{A}(c^*)=\mathcal{A}(c_0)$ with $c_0\in int\mathbb{E}_i$. For any curve $\zeta$ contained in $\mathcal{A}(c^*)\backslash\mathcal{A}(c_0)$, it follows from the formulation
$$
0=\int (L(d\zeta(t))-\langle c^*,\dot\zeta\rangle)dt=\int (L(d\zeta(t))-\langle c,\dot\zeta\rangle)dt+\langle c-c^*,[\zeta]\rangle
$$
that $\langle c-c^*,[\zeta]\rangle\neq0$ holds. We claim $[\zeta]\neq g(\mathbb{E}_i)$. Let us assume the contrary and consider the case that $\zeta$ is a homoclinic curve and $\mathcal{A}(c)$ contains a homocilinic curve $\gamma$. In this case, by assuming that $\alpha(c)=0$ for $c\in\mathbb{F}_0$, we have
$$
\int_{-\infty}^{\infty} L(d\zeta)dt -\langle c^*,[\zeta]\rangle=0, \qquad \int_{-\infty}^{\infty} L(d\gamma)dt -\langle c,g(\mathbb{E}_i)\rangle=0.
$$
Since the class $c^*$ is not on the straight line containing $\mathbb{E}_i$, we have $\langle c^*-c,g(\mathbb{E}_i)\rangle\ne0$. If $\langle c^*-c,g(\mathbb{E}_i)\rangle>0$ we would have
$$
\int_{-\infty}^{\infty} L(d\gamma)dt-\langle c^*,[\gamma]\rangle=\int_{-\infty}^{\infty} L(d\gamma)dt-\langle c,[\gamma]\rangle- \langle c^*-c,g(\mathbb{E}_i)\rangle<0
$$
If $[\zeta]=g(\mathbb{E}_i)$ and $\langle c^*-c,g(\mathbb{E}_i)\rangle<0$ we would have
$$
\int_{-\infty}^{\infty} L(d\zeta)dt-\langle c,[\zeta]\rangle=\int_{-\infty}^{\infty} L(d\gamma)dt-\langle c^*,[\zeta]\rangle+ \langle c^*-c,g(\mathbb{E}_i)\rangle<0
$$
Both cases are absurd as $\alpha(c)=\alpha(c^*)=0$. Because $[\zeta]\neq g(\mathbb{E}_i)$, some $x^*\in\mathcal{A}(c^*)$ remains far away from $\mathcal{A}(c)$. Let $c^*\to c'$, the accumulation point of these points does not fall into $\mathcal{A}(c)$, it implies  $\mathcal{A}(c')\supsetneq\mathcal{A}(c)$. The proof is similar if $\xi$ as well as $\gamma$ is a curve lying in the Mather set.
\end{proof}

Recall the definition of $G_m$ in the section 2: a first homology class $g\in G_m$ if and only if there exists a minimal homoclinic orbit $d\gamma$ such that $[\gamma]=g$. Let $G_{m,c}\subset G_m$ be defined such that $g\in G_{m,c}$ if and only if there exists a minimal homoclinic orbit $d\gamma$ in $\tilde{\mathcal{A}}(c)$ such that $[\gamma]=g$. We say that there are $k$-types of minimal homoclinic orbits in $\tilde{\mathcal{A}}(c)$ if $G_{m,c}$ contains exactly $k$ elements. For an edge we define $G_{m,\mathbb{E}_i}=G_{m,c}$ for each $c\in\mathrm{int} \mathbb{E}_i$, from the proof of Lemma \ref{flatlem2} one can see that it makes sense.
\begin{theo}\label{flatthm3}
Let $\mathbb{F}_0$ be a two dimensional flat, $\mathcal{M}(c_0)$ is a singleton for $c_0\in int\mathbb{F}_0$. Let $\mathbb{E}_i$ denote an edge of $\mathbb{F}_0$ $($not a point$)$, then

1, either $\mathbb{E}_i\cap\partial^*\mathbb{F}_0 =\varnothing$ or $\mathbb{E}_i\subset\partial^*\mathbb{F}_0$;

2, if $\mathbb{E}_i\cap\partial^*\mathbb{F}_0 =\varnothing$, then $G_{m,\mathbb{E}_i}$ contains exactly one element, if $\mathbb{E}_i\subset\partial^*\mathbb{F}_0$, all curves in $\mathcal{M}(\mathbb{E}_i)\backslash\{0\}$ have the same rotation vector;

3, if $c\in\partial\mathbb{E}_i$ and $c\notin\partial^*\mathbb{F}_0$ then $G_{m,c}$ contains exactly two elements;

4, if $\mathbb{E}_i, \mathbb{E}_j\subset\partial^*\mathbb{F}_0$, then either $\mathbb{E}_i$ and $\mathbb{E}_j$ are disjoint, or $\mathbb{E}_i=\mathbb{E}_j$;

5, if $\mathbb{E}_i\subset\partial^*\mathbb{F}_0$, $\mathcal{M}(c)=\mathcal{M}(c')$ holds for $c\in\partial \mathbb{E}_i$ and $c'\in int \mathbb{E}_i$.
\end{theo}
\begin{proof}
For the conclusion 1, as $\mathcal{A}(c)=\mathcal{A}(c')$ if $c,c'\in\mathrm{int}\mathbb{E}_i$ \cite{Mas}, we only need to consider $c\in\partial\mathbb{E}_i$. If it is not true, there would exist an invariant measure $\mu_c$, not supported on the singleton and minimizing the action
$$
\int Ld\mu_c-\langle \rho(\mu_c),c\rangle=-\alpha(c),
$$
but not minimizing the $c'$-action for $c'\in\text{\rm int}\mathbb{E}_i$. As the configuration space is $\mathbb{T}^2$, the Lipschitz graph property of Aubry set will be violated if the rotation vector of the measure $\rho(\mu_c)$ is not parallel to $g\in G_{m,\mathbb{E}_i}$. So, $\langle\rho(\mu_c), c-c'\rangle=0$ holds for  $c'\in\text{\rm int}\mathbb{E}_i$, thus $\mu_c$ also minimizes the action for $c'\in\text{\rm int}\mathbb{E}_i$. This leads to a contradiction. Since $\partial^*\mathbb{F}_0$ is closed, once $\text{\rm int}\mathbb{E}_i\subset\partial^*\mathbb{F}_0$, then whole edge is also contained in $\partial^*\mathbb{F}_0$.

The conclusion 2 follows from the fact that $\langle c-c',[\gamma]\rangle=0$ holds for any $c,c'\in\text{\rm int}\mathbb{E}_i$ and any $\gamma\in\mathcal{A}(c)$, the conclusion 3 follows from that $\mathcal{A}(c)\varsupsetneq\mathcal{A}(c')$ if $c'\in\text{\rm int}\mathbb{E}_i$.

If the conclusion 4 was not true, for the cohomology class in $\mathbb{E}_i\cap \mathbb{E}_j$ the Mather set would contain two closed circles with different homology, but it violates the Lipschitz graph property of Aubry set. With the same reason we have the conclusion 5.
\end{proof}

By this theorem, each edge $\mathbb{E}_i\subset\partial^*\mathbb{F}_0$ aslo uniquely determines a class $g(\mathbb{E}_i)$ so that for each $c\in\mathrm{int}\mathbb{E}_i$, the rotation vector of each $c$-minimal measure has the form $\nu g(\mathbb{E}_i)$ ($\nu>0$). For brevity, we also use the notation $\mathcal{M}(\mathbb{E}_i)= \mathcal{M}(c)$ for $c\in \mathbb{E}_i$.

\begin{figure}[htp] 
  \centering
  \includegraphics[width=6.6cm,height=3.2cm]{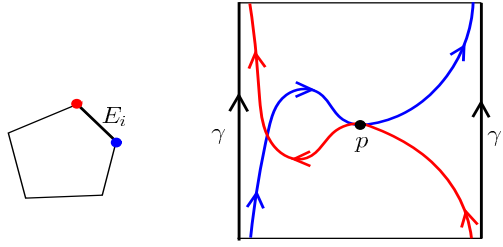}
  \caption{$\mathbb{E}_i\subset\partial^*\mathbb{F}_0$, $\mathcal{M}(\mathbb{E}_i)=\{0\}\cup
  \{\gamma\}$. The blue curve is in $\mathcal{A}(c)$ for $c$ at one end point
  of $\mathbb{E}_i$, the red curve is in $\mathcal{A}(c')$ for $c'$ at another end
  point of $\mathbb{E}_i$.}
  \label{Fig1}
\end{figure}
Given two homology classes $g,g'\in H_1(\mathbb{T}^2,\mathbb{Z})$, we call them adjacent if $g\in G_{m,\mathbb{E}}$, $g'\in G_{m,\mathbb{E}'}$, $\mathbb{E}\cap\partial^*\mathbb{F}_0=\varnothing$, $\mathbb{E}'\cap\partial^*\mathbb{F}_0=\varnothing$, $\mathbb{E}$ and $\mathbb{E}'$ are adjacent. The special topology of two-dimensional torus induces some restrictions on adjacent homologies.
\begin{lem}\label{flatlem3}
Let $\mathbb{E},\mathbb{E}'\subset\partial\mathbb{F}_0\backslash\partial^*\mathbb{F}_0$ be two adjacent edges and assume $c\in \mathbb{E}\cap \mathbb{E}'$. If $(m,n)=g\in G_{m,\mathbb{E}}$ and $(m',n')=g'\in G_{m,\mathbb{E}'}$, then one has that $m'n-mn'=\pm 1$.
\end{lem}
\begin{proof}
The Aubry set $\tilde{\mathcal{A}}(c)$ contains homoclinic orbits with two two classes $(m,n)$ and $(m',n')$, both are irreducible. Guaranteed by the Lipschitz graph property, these curves intersect each other only at the fixed point. In the universal covering space $\mathbb{R}^2$, each curve in the lift of the homoclinic curves are determined by the equation
$$
mx_1+nx_2=k,\qquad m'x_1+n'x_2=k'.
$$
The solution of the equations corresponds to the intersection point which are lattice points in $\mathbb{Z}^2$ for any $(k,k')\in\mathbb{Z}^2$. To guarantee this property, the necessary and sufficient condition is $mn'-m'n=\pm 1$.
\end{proof}

For each indivisible homological class $0\neq g\in H_1(\mathbb{T}^2,\mathbb{Z})$, either $\mathscr{L}_{\beta}(\lambda g)\notin\partial\mathbb{F}_0$ for any $\nu>0$, or some $\lambda_0>0$ exists such that $\mathscr{L}_{\beta}(\lambda_0 g)\in\partial^*\mathbb{F}_0$.

In the first case, $\mathscr{L}_{\beta}(\lambda g)\to\partial\mathbb{F}_0\backslash\partial ^*\mathbb{F}_0$ as $\lambda\downarrow 0$, at least one periodic curve $\gamma_{\lambda}\subset\mathcal{M}(c)$ exists for $c\in\mathscr{L}_{\beta}(\lambda g)$ with $\lambda>0$. It is impossible that $d(c,\mathscr{L}_{\beta}(\lambda g))\to 0$ holds for $c\in\partial^*\mathbb{F}_0$, as in that case certain $c$-minimal measure $\mu_c$ would exist so that $\rho(\mu_c)$ is not parallel to $[\gamma]$. It will violate the Lipschitz property. Generically, $(\gamma_{\lambda},\dot\gamma_{\lambda})$ is hyperbolic and $\mathscr{L}_{\beta} (\lambda g)$ is an interval if $\lambda>0$. If $g\in G_{m,\mathbb{E}_i}$, then $\mathscr{L}_{\beta}(\lambda g)$ approaches to certain edge $\mathbb{E}_i$. If $g=k_ig_i+k_{i+1}g_{i+1}$ with indivisible $(k_i,k_{i+1})\in\mathbb{Z}^2_+$, $g_i\in G_{m,\mathbb{E}_i}$, $g_{i+1}\in G_{m,\mathbb{E}_{i+1}}$, $\mathbb{E}_i$ and $\mathbb{E}_{i+1}$ are two adjacent edges. When $\lambda\downarrow 0$, the interval will shrink to a vertex where $\mathbb{E}_i$ is joined to $\mathbb{E}_{i+1}$, and we have a sequence of closed orbits $\{d\gamma_{\lambda}\} =\{\cup_t(\gamma_{\lambda}(t),\dot\gamma_{\lambda}(t))\}$. Its Kuratowski upper limit set is obviously in the Aubry set for certain $c\in\partial\mathbb{F}_0\backslash\partial^*\mathbb{F}_0 $, thus, consists of minimal homoclinic orbits to the fixed point. As $c$ approaches the vertex, the Mather set will get close to a set of figure eight:
$$
\mathcal{M}(c)\to \gamma_i\ast\gamma_{i+1},
$$
where $\gamma_{\ell}\subset\mathcal{A}(E_{\ell})$ is a minimal homoclinic orbit such that $[\gamma_{\ell}]=g_{\ell}$ for ${\ell}=i,j$.
\begin{figure}[htp]
  \centering
  \includegraphics[width=7cm,height=3.0cm]{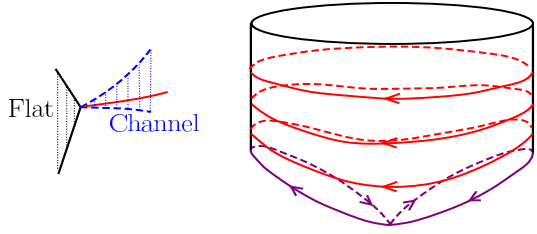}
  \caption{For each $c$ on the red line, the Aubry set is a closed orbit on the cylinder, it approaches a curve of figure eight, $k_i$-folded over $\gamma_i$ and $k_{i+1}$-folded over $\gamma_{i+1}$.}
  \label{fig2}
\end{figure}

To be more precise, let us consider it in the finite covering space $\bar M=\bar k_1\mathbb{T} \times\bar k_2\mathbb{T}$ where $\bar k_m=k_ig_{im}+k_{i+1}g_{i+1,m}$ for $m=1,2$ if we write $g_j=(g_{j1},g_{j2})$ for $j=i,i+1$. Let $\sigma$: $\{1,2,\cdots,k_i+k_{i+1}\}\to \{i,{i+1}\}$ be a permutation such that the cardinality $\#\sigma^{-1}(i)=k_i$ and $\#\sigma^{-1}({i+1})=k_{i+1}$. The lift of homoclinic curve $\gamma_i$ as well as $\gamma_{i+1}$ to $\bar M$ contains several curves. Pick up one curve $\bar\gamma_{\sigma(1)}$ in the lift of $\gamma_{\sigma(1)}$, it determines a unique curve $\bar\gamma_{\sigma(2)}$ such that the end point of $\bar\gamma_{\sigma(1)}$ is the starting point of $\bar\gamma_{\sigma(2)}$, and so on.
We shall see that there exists a unique permutation $\sigma$ (up to a translation) such that, as $c$ approaches the vertex through in channel, each Aubry class in $\mathcal{A}(c,\bar M)$ approaches (up to a translation) the curve $\bar\gamma_{\sigma(1)}\ast\bar\gamma_{\sigma(2)} \ast\cdots\ast\bar\gamma_{\sigma(k_i+k_{i+1})}$ without folding. Here $\mathcal{A}(c,\bar M)$ denotes the Aubry set with respect to $\bar M$. Although $\mathcal{A}(c)$ is made up by one periodic curve, $\mathcal{A}(c,\bar M)$ may contain several periodic curves.

As the minimal curve $\gamma_{\lambda}$ is periodic with the homological class $[\gamma_{\lambda}]=k_ig_i+k_{i+1}g_{i+1}$, the permutation $\sigma$: $\mathbb{Z}\to \{i,i+1\}$ is $(k_i+k_{i+1})$-periodic. Since $k_i$ is prime to $k_{i+1}$, we have $k_i=k_{i+1}=1$ if $k_i=k_{i+1}$.
\begin{lem}\label{flatlem4}
The permutation is uniquely $($up to a translation$)$ determined by $k_i$ and $k_{i+1}$. If $k_{i}>k_{i+1}$, the following holds for $j=1,\cdots,k_i+k_{i+1}$
\begin{align*}
&\sigma(j+j_0)=i,  \hskip 1.0 true cm \text{\rm if}\ \ j=2,4,\cdots,2k_{i+1}, 2k_{i+1}+1,\cdots,k_{i}+k_{i+1},\\
&\sigma(j+j_0)=i+1,\hskip 0.35 true cm \text{\rm if}\ \ j=1,3,\cdots,2k_{i+1}-1.
\end{align*}
\end{lem}
\begin{proof}
By the assumption, there exists only one minimal homoclinic curve $\gamma_j$ such that $[\gamma_j]=g_j$ for $j=i,i+1$. Because of the lemma \ref{flatlem3}, we can assume that $g_i=(1,0)$ and $g_{i+1}=(0,1)$ by choosing suitable coordinates on $\mathbb{T}^2$. We choose two sections $I^-$ and $I^+$ in a small neighborhood of the origin such that, emanating from the origin, these homoclinic curves pass through $I^-$ and $I^+$ successively before they return back to the origin as $t\to\infty$, see the figure below. In the section $I^{\pm}$ we choose disjoint subsections $I^{\pm}_i$ and $I^{\pm}_{i+1}$ such that the curve $\gamma_j$ passes through $I^{\pm}_{j}$ for $j=i,i+1$.

Let $\gamma_{\lambda}$ be the minimal periodic curve with rotation vector $\lambda g$. For small $\lambda>0$, $\gamma_{\lambda}$ falls into a small neighborhood of these two homoclinic curves. So it has to pass either through $I^{\pm}_i$ or through $I^{\pm}_{i+1}$. Let $t_{\ell}^{\pm}$ be the time for $\gamma_{\lambda}$ passing through $I^{\pm}$ with $\cdots<t_{{\ell}-1}^-<t_{\ell}^+<t_{\ell}^-<t_{{\ell}+1}^+<\cdots$, and it does not tough these sections whenever $t\ne t_k^{\pm}$. By definition, the period of the curve equals $t_{k_1+k_2}^{\pm}-t_0^{\pm}$.   If the curve intersects $I^+_i$ at $t^+_{\ell}$ and intersects $I^-_{i+1}$ at $t^-_{\ell}$, then the segment $\gamma_{\lambda}|_{[t_{{\ell}-1}^-,t_{\ell}^+]}$ keeps close to $\gamma_i$ and $\gamma_{\lambda}|_{[t_{\ell}^-,t_{\ell+1}^+]}$ keeps close to $\gamma_{i+1}$, so one has $\gamma_{\lambda}(t^-_{\ell-1})\in I_i^-$ and $\gamma_{\lambda}(t^+_{\ell+1})\in I^+_{i+1}$.
\begin{figure}[htp]
  \centering
  \includegraphics[width=7.8cm,height=3.5cm]{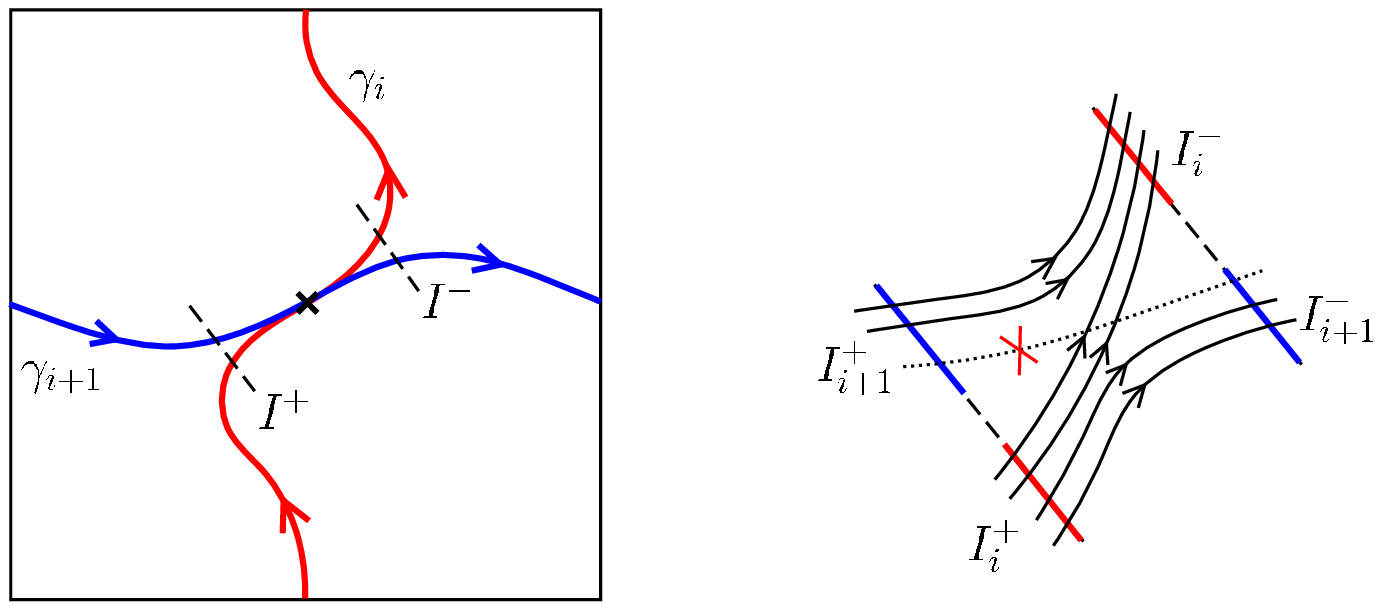}
  \label{fig4}
\end{figure}
Since the curve $\gamma_{\lambda}$ is minimal, it does not have self-intersection. Thus, once there exists $t_j^{\pm}$ such that $\gamma_{\lambda}(t^+_j)\in I^+_i$ and $\gamma_{\lambda}(t^-_j)\in I^-_i$, then there does not exist $t_{j'}^{\pm}$ such that $\gamma_{\lambda}(t^+_{j'})\in I^+_{i+1}$ and $\gamma_{\lambda}(t^-_{j'})\in I^-_{i+1}$. Therefore, there is a set $J\subset\{1,2,\cdots,k_i+k_{i+1}\}$ with cardinality $\#(J)=k_i-k_{i+1}$ such that for $j\in J$ one has $\gamma_{\lambda}(t^{\pm}_{j})\in I^{\pm}_{i}$, for $j\notin J$ one either has $\gamma_{\lambda}(t^{+}_{j})\in I^{+}_{i}$ and $\gamma_{\lambda}(t^{-}_{j})\in I^{-}_{i+1}$ or has $\gamma_{\lambda}(t^{+}_{j})\in I^{+}_{i+1}$ and $\gamma_{\lambda}(t^{-}_{j})\in I^{-}_{i}$.

By introducing coordinate transformation on $T$: $\mathbb{T}^2\to\mathbb{T}^2$ such that  $T_*g=g$ $\forall\,g\in H_1(\mathbb{T}^2,\mathbb{Z})$, let us think the curve $T\gamma_{\lambda}$ as a straight line projected down to the unit square, a fundamental domain of $\mathbb{T}^2$.  Starting from a point $z^h_0=(x_0,0)$, the line successively reaches to the points $z^h_1=(x_1,0),\cdots,z^h_m=(x_m,0),\cdots,z^h_{k_i}=z^h_{0}$ where $x_m=(x_0+mk_{i+1}/k_i\mod 1,0)$ with small $x_0>0$. To connect the point $(x_{m-1},0)$ to the point $(x_m,1)$, the curve $T\gamma_{\lambda}$ does not touch the vertical boundary lines if
$$
\Big[(m-1)\frac{k_{i+1}}{k_i}\Big]=\Big[m\frac{k_{i+1}}{k_i}\Big],
$$
where $[a]$ denote the largest integer not bigger than the number $a$, and it has to pass through the vertical lines at some point $z^v_m=(0\mod 1,y_m)$ if
$$
\Big[(m-1)\frac{k_{i+1}}{k_i}\Big]+1=\Big[m\frac{k_{i+1}}{k_i}\Big].
$$
We define an order $\prec$ for these $k_i+k_{i+1}$ points such that $z^h_j\prec z^h_k$ iff $j<k$ and $z^h_j\prec z^v_{j+1}\prec z^h_{j+1}$ iff $[jk_i/k_{j+1}]+1=[(j+1)k_i/k_{j+1}]$.

Returning back to the original coordinates, the curve $\gamma_{\lambda}$ falls into a neighborhood of the curves $\gamma_{i}$ and $\gamma_{i+1}$, intersects the horizontal line $\Gamma_h=T^{-1}\{(x_1,x_2):x_1=\frac12\mod 1\}$ at $T^{-1}z_j^h$ and intersects the vertical line $\Gamma_v=T^{-1}\{(x_1,x_2):x_2=\frac12\mod 1\}$ at $T^{-1}z_j^v$, $[\Gamma_h]=g_{i+1}$ and $[\Gamma_v]=g_i$. Naturally, the map $T$ induces the order among these points: $T^{-1}z^{h,v}_j\prec T^{-1}z^{h,v}_{\ell}$ if and only if $z^{h,v}_j\prec z^{h,v}_{\ell}$. If the curve passes the point $T^{-1}z^h_j$ at $t\in(t^-_j,t^+_{j+1})$, the segment $\gamma_{\lambda}|_{[t^-_j,t^+_{j+1}]}$ falls into a neighborhood of $\gamma_i$, otherwise, it falls into a neighborhood of $\gamma_{i+1}$. In this way, we obtained a unique permutation $\sigma$ up to a translation.
\begin{figure}[htp]
  \centering
  \includegraphics[width=7cm,height=3.5cm]{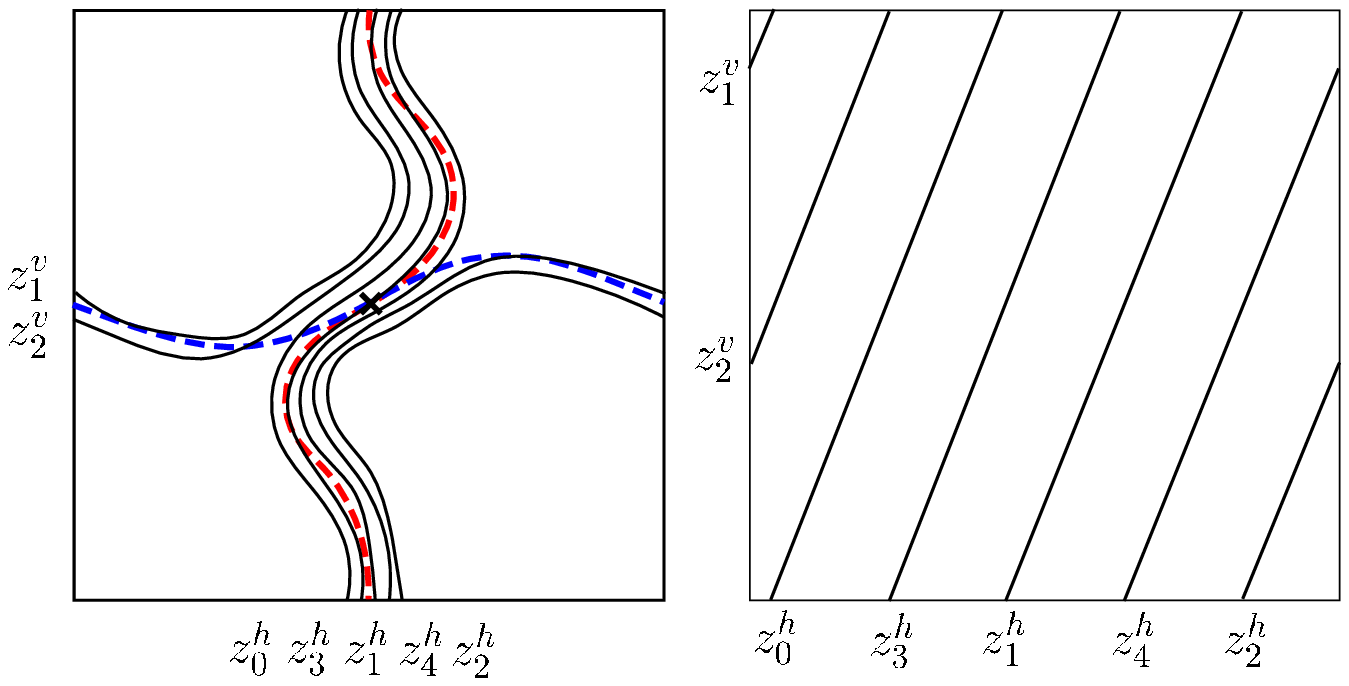}
  \label{fig5}
\end{figure}
\end{proof}

\section{Cohomology equivalence around the flat $\mathbb{F}_0$}
\setcounter{equation}{0}
Let us start with the Hamiltonian $G_{\epsilon}$ defined by Formula (\ref{mainsystems}). Given two homology classes $g,g'\in H_1(\mathbb{T}^2,\mathbb{Z})$, Theorem 1.1 of \cite{C15} confirms the existence of two wedge-shaped channels $\mathbb{W}$ and $\mathbb{W}'$ which reach to the boundary of the annulus
$$
\mathbb{A}_0=\Big\{c\in H^1(M,\mathbb{R}):0<\alpha_{G_{\epsilon}}(c)-\min\alpha_{G_{\epsilon}}<D\epsilon^{d}
\Big\},
$$
Each channel corresponds to a normally hyperbolic invariant cylinder (NHIC), for each cohomology class in the channel, the Aubry set lies on the NHIC. However, these channels may not reach to the flat $\mathbb{F}_0$. Therefore, a notable difficulty arises that these cylinders are separated by an annulus $\mathbb{A}_0$ and how to connect them by a transition chain? It is the problem of crossing double resonance.

It is the goal of this section to find an annulus $\mathbb{A}\supsetneq\mathbb{A}_0$ such that both $\mathbb{W}$ and $\mathbb{W}'$ extend into. This annulus admits a foliation of circles of cohomology equivalence, so these two channels are connected by transition chains of cohomology equivalence.

\subsection{Cohomology equivalence for autonomous systems}
It is a new version of cohomology equivalence we used in \cite{LC}. The concept of cohomology equivalence was introduced in \cite{M93} for the first time, but it does not apply in interesting problems of autonomous system. The new version is defined not on the whole configuration space, but on a section. For $n$-torus, the section is chosen as non-degenerately embedded section $(n-1)$-dimensional torus. We call $\Sigma_c$ non-degenerately embedded  ($n-1$)-dimensional torus by assuming a smooth injection $\varphi$: $\mathbb{T}^{n-1}\to\mathbb{T}^n$ such that $\Sigma_c$ is the image of $\varphi$, and the induced map $\varphi_*$: $H_1(\mathbb{T}^{n-1}, \mathbb{Z})\to H_1(\mathbb{T}^{n},\mathbb{Z})$ is an injection.

Let $\mathfrak{C}\subset H^1(\mathbb{T}^n,\mathbb{R})$ be a connected set where we are going to define $c$-equivalence. For each class $c\in \mathfrak{C} $, we assume that there exists a non-degenerate embedded $(n-1)$-dimensional torus $\Sigma_c\subset\mathbb{T}^n$ such that each $c$-semi static curve $\gamma$ transversally intersects $\Sigma_c$. Let
$$
\mathbb{V}_{c}=\bigcap_U\{i_{U*}H_1(U,\mathbb{R}): U\, \text{\rm is a neighborhood of}\, \mathcal {N}(c) \cap\Sigma_c\},
$$
here $i_U$: $U\to M$ denotes inclusion map. $\mathbb{V}_{c}^{\bot}$ is defined to be the annihilator of $\mathbb{V}_{c}$, i.e. if $c'\in H^1(\mathbb{T}^n,\mathbb{R})$, then $c'\in \mathbb{V}_{c}^{\bot}$ if and only if $\langle c',h \rangle =0$ for all $h\in \mathbb{V}_c$. Clearly,
$$
\mathbb{V}_{c}^{\bot}=\bigcup_U\{\text{\rm ker}\, i_{U}^*: U\, \text{\rm is a neighborhood of}\, \mathcal {N}(c) \cap\Sigma_c\}.
$$
Note that there exists a neighborhood $U$ of $\mathcal {N}(c)\cap\Sigma_c$ such that $\mathbb{V}_c=i_{U*}H_1(U,\mathbb{R})$ and $\mathbb{V}_{c}^{\bot}=\text{\rm ker}i^*_U$ (see \cite{M93}).

We say that $c,c'\in H^1(M,\mathbb{R})$ are $c$-equivalent if there exists a continuous curve $\Gamma$: $[0,1]\to \mathfrak{C}$ such that $\Gamma(0)=c$, $\Gamma(1)=c'$, $\alpha(\Gamma(s))$ keeps constant along $\Gamma$, and for each $s_0\in [0,1]$ there exists $\delta>0$ such that $\Gamma(s)-\Gamma(s_0)\in \mathbb{V}_{{\Gamma}(s_0)}^{\bot}$ whenever $s\in [0,1]$ and $|s-s_0|<\delta$.

Let $\{e_i\}_{1\le i\le n-1}$ be the standard basis of $H_1(\mathbb{T}^{n-1}, \mathbb{Z})$, one obtains $n$-dimensional vectors $\{g_{i+1}=\varphi_*(e_i)\in H_1(\mathbb{T}^{n}, \mathbb{Z})\}_{1\le i\le n-1}$. Because $\varphi$ is injection, there is a vector $g_1\in\mathbb{Z}^n$ such that the $n\times n$ matrix $G=(g_1,g_2,\cdots, g_n)$ is uni-module, i.e. $\text{\rm det}G=\pm 1$. In new coordinates system $x\to G^{-1}x$, the Lagrangian $\tilde L(\dot x,x)=L(G\dot x,Gx)$ is also $2\pi$-periodic in $x$. In new coordinates, let $\bar M=\mathbb{R}\times\mathbb{T}^{n-1}= \{x_1\in\mathbb{R},(x_2,\cdots,x_n)\in \mathbb{T}^{n-1}\}$ be a covering space of $\mathbb{T}^n$, $\pi:\bar M\to M=\mathbb{T}^n$. The lift of $\Sigma_c$, $\pi^{-1}(\Sigma_c)$ has infinitely many compact components $\{\Sigma_c^i\}_{i\in\mathbb{Z}}$. If $\varphi$ is linear, $\Sigma_c^i=\{x_1=2i\pi\}$. For the section $\Sigma=\{x_1=0\mod 1\}$ we have $\pi^{-1}(\Sigma)=\cup_{k\in\mathbb{Z}}\{x_1=k\}$ while for $\Sigma=\{x_1=x_2\}$, the lift $\pi^{-1}(\Sigma)$ consists of only one connected component.

\subsection{Dynamics around fixed point}

Before involved into the details how such cohomology equivalence is found, let us establish a lemma which is technically crucial for the follow-up demonstration. For $\theta>0$ and $\xi\in\mathbb{R}^n\backslash\{0\}$, we define a cone
$$
C(\xi,\theta,d)=\{x\in\mathbb{R}^n: |\langle x,\xi\rangle\ge\theta\|\xi\|\|x\|,\|x\|=d\}.
$$
\begin{lem}\label{lemma3.1}
Assume that $(x,y)=(0,0)\in\{ H^{-1}(0)\}$ is a hyperbolic fixed point for $\Phi_H^t$, where all eigenvalues are real and different:
$$
\mathrm{Spec}\{J\nabla H\}=\{\pm\lambda_1,\cdots,\pm\lambda_n; 0<\lambda_1<\cdots<\lambda_n\}.
$$
Let $(\xi^{\pm}_i,\eta^{\pm}_i)$ denote the eigenvector for $\pm\lambda_i$, where $\xi^{\pm}_i$ is for the $x$-coordinates, $\eta^{\pm}_i$ is for the $y$-coordinates. Let $(x(t),y(t))\subset\{H^{-1}(0)\}$ be an orbit such that $x(t)$ passes through a ball $B_{\delta}(0)\subset\mathbb{R}^n$ such that $x(-T)\in\partial B_{\delta}(0)$, $x(T)\in\partial B_{\delta}(0)$ and $x(t)\in\mathrm{int}B_{\delta}(0)$ $\forall$ $t\in(-T,T)$. Then, for small $\delta>0$ and $\theta\ge\frac 12$, there exist sufficiently large $T_{\theta}>0$ such that for $T\ge T_{\theta}$ one has
$$
(x(-T),x(T))\notin C(\xi^+_1,\theta,\delta)\times C(\xi^-_1,\theta,\delta).
$$
If $x(\pm T)\in C(\xi^{\pm}_1,\theta,\delta)$, for $T\to\infty$, then, there exist constant $d^{\mp}_2,\cdots,d^{\mp}_n$ such that
$$
\Big|x(\mp T)-\sum_{i=2}^nd^{\mp}_i\xi^{\mp}_i\Big|=o(\delta).
$$
\end{lem}
\begin{proof}
By certain symplectic coordinate transformation, the Hamiltonian is assumed to have the normal form
$$
H(x,y)=\sum_{i=1}^n\frac 12\Big(y_i^2-\lambda^2_ix^2_i\Big)+P_3(x,y)
$$
where $P_3(x,y)=O(\|(x,y)\|^3)$ is a higher order term. By applying Hartman-Grobman Theorem we see that the Hamiltonian flow is $C^0$-conjugate to the linear flow
\begin{equation}\label{equation3.1}
\begin{aligned}
\bar x_i(t)=&b^-_ie^{-\lambda_it}+b^+_ie^{\lambda_it},\\
\bar y_i(t)=&\lambda_i(-b^-_ie^{-\lambda_it}+b^+_ie^{\lambda_it}).
\end{aligned}
\end{equation}
From the proof of the theorem one can see that, restricted in the ball $B_{\delta}(0)$, the conjugacy is $O({\delta})$-close to identity. Therefore, the solution of the Hamiltonian equation takes the form
\begin{equation}\label{equation3.2}
\begin{aligned}
x_i(t)=&b^-_ie^{-\lambda_it}+b^+_ie^{\lambda_it}+o(\delta),\\
y_i(t)=&\lambda_i(-b^-_ie^{-\lambda_it}+b^+_ie^{\lambda_it})+o(\delta),
\end{aligned}
\end{equation}
provided $|x(t)|,|y(t)|\le\delta$ holds for $|t|\le t_0$. Given a boundary condition $x(-T),x(T)$ such that $|x(\pm T)|=\delta$ with small $\delta$, there is a unique solution $x(t),y(t)$ such that $|x(t)|,|y(t)|\le\delta$ for $t\in [-T,T]$ provided $T>0$ is suitably large.

For $\bar x(\pm T)=\delta$ with $T\gg 1$, then $|\bar x(0)|, |\bar y(0)|\le O(\delta e^{-\lambda_1T})$. In a ball of diameter of order $O(\delta e^{-\lambda_1T})$ centered at the origin, the conjugacy is $O(\delta e^{-\lambda_1T})$-close to identity, so we have
\begin{equation}\label{equation3.3}
\begin{aligned}
x_i(0)=&b^-_i+b^+_i+o(\delta e^{-\lambda_1T}),\\
y_i(0)=&\lambda_i(b^+_i-b^-_i)+o(\delta e^{-\lambda_1T}).
\end{aligned}
\end{equation}
Substituting $(x,y)$ with the formula (\ref{equation3.3}) in the Hamiltonian we obtain a constraint for the constants $b^{\pm}_i$:
\begin{equation}\label{equation3.4}
\begin{aligned}
H(x(t),y(t))=&-2\sum_{i=1}^n\lambda^2_ib^-_ib^+_i+(b^+_i\pm b^-_i)o(\delta e^{-\lambda_1T})+o(\delta e^{-\lambda_1T})^2\\
&+P_3(b^+_i+ b^-_i+o(\delta e^{-\lambda_1T}),\lambda_i(b^+_i-b^-_i)+o(\delta e^{-\lambda_1T})).
\end{aligned}
\end{equation}
Let us investigate how the constants $b^{\pm}_i$ depend on the boundary condition $x(T)=(x^+_1,x^+_2,\cdots,x^+_n)\in\partial B_{\delta}(0)$, $x(-T)=(x^-_1,x^-_2,\cdots,x^-_n)\in\partial B_{\delta}(0)$ by assuming
\begin{equation}\label{equation3.5}
\min\{|x^-_1|,|x^+_1|\}\ge\frac{\delta}2.
\end{equation}

For $\theta=1/2$, $(x(-T),x(T))\in C(\xi^-_1,\theta,\delta)\times C(\xi^+_1,\theta,\delta)$ implies (\ref{equation3.5}) holds. Because the curve $x|_{[-T,T]}$ stays inside of the ball $B_{\delta}(0)$ and $T$ is sufficiently large, the orbit $(x,y)|_{[-T,T]}$ stays near the stable and unstable manifold of the fixed point. From (\ref{equation3.2}) one immediately obtains the solution
\begin{equation}\label{equation3.6}
\begin{aligned}
b^-_i=&e^{-\lambda_iT}(x^-_i+o(\delta))(1+O(e^{-\lambda_1T})),\\
b^+_i=&e^{-\lambda_iT}(x^+_i+o(\delta))(1+O(e^{-\lambda_1T})).
\end{aligned}
\end{equation}
For sufficiently large $T>0$, it follows from the assumption (\ref{equation3.5}) and the equation (\ref{equation3.6}) that
$$
|b^{\pm}_1|\ge\frac{\delta}3e^{-\lambda_1T},
$$
and
$$
|b^{\pm}_i|\le 2\delta e^{-\lambda_iT}, \qquad \forall\ i=2,\cdots,n.
$$
Since $\lambda_1<\lambda_i$ for each $i\ge 2$, $|b^{\pm}_i|\ll |b^{\pm}_1|$ if $T$ is sufficiently large. In this case, we obtain from (\ref{equation3.4}) that
$$
|H(x(t),y(t))|>|\lambda^2_1b^+_1b^-_1|>0.
$$
It contradicts the assumption that $(x(t),y(t))\in\{H^{-1}(0)\}$. Let $T\to\infty$, one easily see the last conclusion.
\end{proof}

Because of the relation between the eigenvalues $|\lambda_1|<|\lambda_2|$, on the unstable manifold there exist exactly two orbits $(\gamma^-(t),\dot\gamma^-(t))$ and $(\gamma'^-(t),\dot\gamma'^-(t))$ which approach the origin as $t\to -\infty$ in the direction of $\Lambda_{2,x}$:
\begin{equation}\label{direction2}
\lim_{t\to -\infty}\frac{\dot\gamma^-(t)}{|\dot\gamma^-(t)|}=\Lambda_{2,x}, \qquad \lim_{t\to -\infty}\frac{\dot\gamma'^-(t)}{|\dot\gamma'^-(t)|}=-\Lambda_{2,x}.
\end{equation}
On the stable manifold there exist exactly two orbits $(\gamma^+(t),\dot\gamma^+(t))$ and $(\gamma'^+(t),\dot\gamma'^+(t))$ which approach the origin as $t\to\infty$ in the direction of $\Lambda_{2,x}$:
\begin{equation}\label{direction2'}
\lim_{t\to\infty}\frac{\dot\gamma^+(t)}{|\dot\gamma^+(t)|}=\Lambda_{2,x}, \qquad \lim_{t\to\infty}\frac{\dot\gamma'^+(t)}{|\dot\gamma'^+(t)|}=-\Lambda_{2,x}.
\end{equation}
These curves intersect the circle $\partial B_{\delta}(0)$ at four points: $x^{\pm}_{\delta}$ is the intersection point of $\gamma^{\pm}$ with $\partial B_{\delta}(0)$ and $x'^{\pm}_{\delta}$ is the intersection point of $\gamma'^{\pm}$ with $\partial B_{\delta}(0)$. Obviously, for small $\delta>0$, $x^+_{\delta}$ is close to $x^-_{\delta}$ and $x'^+_{\delta}$ is close to $x'^-_{\delta}$.

\subsection{The Ma\~n\'e set for $c\in\partial^* \mathbb{F}_0$: irrational case}

As the first step, let us consider the Hamiltonian $G$ defined in (\ref{averagesystem}) and study all cases when the Ma\~n\'e set covers the whole configuration space. We are going to prove
\begin{theo}\label{th3.1}
For Hamiltonian of $($\ref{averagesystem}$)$ with the hypothesis $(${\bf H1}$)$, $c\in\partial^*\mathbb{F}_0$ such that $\rho(\mu_c)$, the rotation vector of $\mu_c$, is irrational and $\mathcal{N}(c)=\mathbb{T}^2$, then there exists $c$-semi static curve which approaches the fixed point in the direction of $\pm\Lambda_{2,x}$ as $t\to\infty$ or $t\to-\infty$.
\end{theo}
\begin{proof}
By the definition, for each class $c\in\partial^* \mathbb{F}_0$, except for a minimal measure $\mu$ supported on the hyperbolic fixed point $(x,\dot x)=0$, some minimal measure $\mu_c$ exists with non-zero rotation vector. For typical potential $V$, there exist at most three ergodic minimal measures for each first cohomology class (see \cite{BC}). For such potential, in the universal covering space $\bar\pi$: $\mathbb{R}^2\to\mathbb{T}^2$, there exists a strip $\mathcal{S}\ni\{x=0\}$, bounded by two $c$-static curves $\xi_c$ and $\xi'_c$ in the sense that $\bar\pi\xi_c, \bar\pi\xi'_c\subset\mathcal{M}(c)$, such that $\mathrm{int}\bar\pi\mathcal{S}_c\cap\mathcal{M}(c)=\{0\}$. If the  Ma\~n\'e set coveres the whole torus $\mathbb{T}^2$, this strip is filled with $c$-semi-static curves $\xi_{\lambda}$ in the sense that $\bar\pi\xi_{\lambda}$ is $c$-semi static, i.e. passing through every point in the strip there is at least one $c$-semi static curve.

\begin{lem}\label{lem3.2}
If the rotation vector of $\mu_c$ is irrational, passing through each point in the strip $\mathcal{S}_c$ there is only one semi-static curve.
\end{lem}
\begin{proof}
As the configuration space is two-torus, two orbits $\bar\pi(\xi_c,\dot\xi_c)$ and $\bar\pi(\xi'_c,\dot\xi'_c)$ share the same rotation vector, where we extend the standard projection $\bar\pi$: $\mathbb{R}^2\to\mathbb{T}^2$ in a natural way to $\bar\pi$: $T\mathbb{R}^2\to T\mathbb{T}^2$, it keeps the velocity unchanged. Properly choosing a section $\Sigma_c$ of $\mathbb{T}^2$ which is a circle, the closure of $\bar\pi\xi_c\cap\Sigma_c$ and of $\bar\pi\xi'_c\cap\Sigma_c$ are Denjoy set. So, two curves $\bar\pi\xi_c$ and $\bar\pi\xi'_c$ are in the same Aubry class, i.e. $h^{\infty}_c(x,x')+h^{\infty}_c(x',x)=0$ holds for any $x\in\bar\pi\xi_c,x'\in\bar\pi\xi'_c$.

If there were two semi-static curves $\gamma_c,\gamma'_c$ intersecting each other at $\gamma_c(0)=\gamma'_c(0)$, the $\omega$-limit set of the orbit $(\gamma_c,\dot\gamma_c)$ must be different from the $\omega$-limit set of the orbit $(\gamma'_c,\dot\gamma'_c)$. It is a consequence of the Lipschitz roperty of Aubry set. Let us assume $\gamma_c(t)\to 0$ and $\gamma'_c(t)\to\bar\pi\xi_c$ as $t\to\infty$. There are four possibilities for $t\to-\infty$
\begin{enumerate}
   \item $\gamma_c(t)\to\{0\}$ and $\gamma'_c(t)\to\bar\pi\xi_c$ as $t\to-\infty$;
   \item $\gamma_c(t)\to\{0\}$ and $\gamma'_c(t)\to\{0\}$ as $t\to-\infty$;
   \item $\gamma_c(t)\to\bar\pi\xi_c$ and $\gamma'_c(t)\to\{0\}$ as $t\to-\infty$;
   \item $\gamma_c(t)\to\bar\pi\xi_c$ and $\gamma'_c(t)\to\bar\pi\xi_c$ as $t\to-\infty$.
\end{enumerate}
In Case 3, we join $\gamma_c(-\epsilon)$ and $\gamma'_c(\epsilon)$ by a minimal curve $\eta$, join $\gamma'_c(-\epsilon)$ and $\gamma_c(\epsilon)$ by a minimal curve $\eta'$. The action along $\eta$ plus it along $\eta'$ is smaller than the action along $\gamma_c|_{[-\epsilon,\epsilon]}$ plus it along $\gamma'_c|_{[-\epsilon,\epsilon]}$. The action along the singular circle $\gamma_c|_{[\epsilon,\infty)}\ast\gamma'_c|_{(-\infty,-\epsilon]}\ast\eta'$ is non-negative. Therefore, the action along the curve $\gamma_c|_{[-K',-\epsilon]}\ast\eta\ast\gamma'_c|_{[\epsilon,K]}$ is smaller than the action along the curve $\gamma_c|_{[-K',0]}\ast\gamma'_c|_{[0,K]}$ for any large $K,K'>0$. It implies that $h^{\infty}_c(x,0)+h^{\infty}_c(0,x')>h^{\infty}(x,x')$ holds for any $x,x'\in\bar\pi\xi_c$. It contradicts the assumption that the Ma\~n\'e set covers the whole torus: for any $m\in\mathbb{T}^2$, $x,x'\in\mathcal{M}(c)$ one has $h^{\infty}_c(x,m)+h^{\infty}_c(m,x')=h^{\infty}(x,x')$.

For Case $4$, there is a semi-static curve $\gamma''_c$ which approaches $\{0\}$ as $t\to-\infty$ and approaches $\xi_c$ as $t\to\infty$. We join $\gamma'_c(-\epsilon)$ to $\gamma_c(\epsilon)$ by a minimal curve $\eta'$, join $\gamma_c(-\epsilon)$ to $\gamma'_c(\epsilon)$ by a minimal curve $\eta$. Given $x,x'\in\bar\pi\xi_c$, $\exists$ sequence $K_i, K'_i\to\infty$ such that $\gamma'_c(-K'_i)\to x$, $\gamma''_c(K_i)\to x'$. Also $\exists$ sequences $N_i,N'_i\to\infty$ such that $\gamma_c(-N'_i)\to x$, $\gamma'_c(N_i)\to x'$. As the action along $\gamma_c|{[-N'_i,-\epsilon]}\ast\eta\ast\gamma'_c|_{[\epsilon,N_i]}$ plus the quantity $h^{\infty}_c(x',x)$ is non-negative, the action along $\gamma'_c|_{[-K',-\epsilon]}\ast\eta'\ast\gamma_c|_{[\epsilon,\infty)}\ast\gamma''_c|_{(-\infty,K]}$ is smaller than the action along $\gamma'_c|_{[-N'_i,\infty)]}\ast\gamma''_c|_{(-\infty,K]}$. Again, it implies that $h^{\infty}_c(x,0)+h^{\infty}_c(0,x')>h^{\infty}(x,x')$ holds for $x,x'\in\bar\pi\xi_c$. It is absurd. Other cases can also be proved by similar argument.
\end{proof}

Next, let us study what will happen if the strip $\mathcal{S}_c$ is filled with semi-static curves when the rotation vector of $\mu_c$ is irrational.

As the curves $\xi_c$ as well as $\xi'_c$ is disjoint with the origin, some number $\delta>0$ exists so that these two curves do not hit the ball $B_{\delta}(0)$. Let $I^{\pm}_{c,\delta}\subset\partial B_{\delta}(0)$ be such a set that passing through each point $x\in I^{\pm}_{c,\delta}$ a $c$-semi static curve approaches to the origin, as $t\to\pm\infty$. Obviously, the set $I^{\pm}_{c,\delta}\ne\varnothing$ is closed and $I^{-}_{c,\delta}\cap I^{+}_{c,\delta}=\varnothing$ (see the proof of Lemma \ref{lem3.2}). Indeed, one has even stringer property as follows:

\begin{lem}\label{lem3.3}
Assume the Ma\~n\'e set covers a neighborhood of the hyperbolic fixed point. Then, some number $\mu>0$ exists so that for suitably small $\delta>0$, the distance between $I^{+}_{c,\delta}$ and $I^{-}_{c,\delta}$ is not smaller than $\mu\delta$.
\end{lem}
\begin{proof}
Let $x^-_{c,\delta}\in I^-_{c,\delta}$ and $x^+_{\\c,\delta}\in I^+_{c,\delta}$ be the endpoint of $I^-_{c,\delta}$ and $I^+_{c,\delta}$ respectively so that $d(x^-_{c,\delta},x^+_{c,\delta})=\mu\delta$. Passing through the point $x^{\pm}_{c,\delta}$ there is a $c$-semi static curve $\gamma^{\pm}_c$ such that $\gamma^{\pm}_c(t)\to 0$ as $t\to\pm\infty$. One then has a wedge-shaped region in $B_{\delta}(0)$, bounded by $\gamma^-_c$ and $\gamma^+_c$ and denoted by $\mathbb{W}$, see the left of Figure \ref{fig8}. No semi-static curve passes through $\mathrm{int}\mathbb{W}$ to approach the origin. Since the fixed point is hyperbolic, it has its stable and unstable manifold, determined by the generating functions $U^+$ and $U^-$ respectively, namely, the stable (unstable) manifold is the graph of the differential of $U^+$ ($U^-$). Restricted in $\mathbb{W}$, these functions satisfy the condition
$$
U^-(x)-U^-(0)\ge\frac{\lambda_1^2}3\|x\|^2, \qquad U^+(0)-U^+(x)\ge \frac{\lambda_1^2}3\|x\|^2,\ \ \ \forall\ \|x\|\le\delta.
$$
Let $x^+_{c,i,\delta}\in\partial B_{\delta}(0)$ be a sequence of points which is located between $x^+_{c,\delta}$ and $x^-_{c,\delta}$ such that $x^+_{c,i,\delta}\to x^+_{c,\delta}$ as $i\to\infty$. By the assumption, passing through $x^+_{c,i,\delta}$ there is a semi static curve $\gamma_{c,i}$ which keeps close to the curve $\gamma^+_{c,\delta}$ before it is going to get close to the origin. After that, it keeps close to the curve $\gamma^-_{c,\delta}$ and intersect the circle $\partial B_{\delta}(0)$ at a point $x^-_{c,i,\delta}$, see the left of Figure \ref{fig8}.
Clearly, $x^-_{c,i,\delta}\to x^-_{c,\delta}$ as $i\to\infty$. By time translation, we assume $\gamma_{c,i}(0)=x^+_{c,i,\delta}$, $\gamma_{c,i}(T_i)=x^-_{c,i,\delta}$.
\begin{figure}[htp]
  \centering
  \includegraphics[width=9.0cm,height=2.5cm]{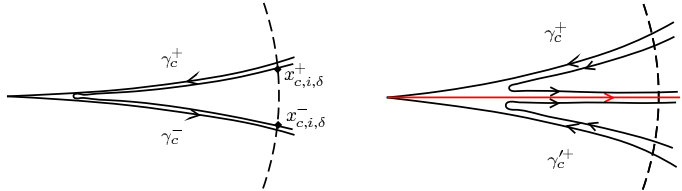}
  \caption{The red line is a semi-static curve departing from the origin.}
  \label{fig8}
\end{figure}
If we set $U^-(0)=U^+(0)$, the action along the curve $\gamma_{c,i}$
$$
A[\gamma_{c,i}|_{[0,T_i]}]\to U^-(x^-_{c,\delta})-U^+(x^+_{c,\delta})\ge\frac{2\lambda_1^2}3\delta^2,\qquad \mathrm{as}\ i\to\infty.
$$
On the other hand, we join the point $x^-_{c,i,\delta}$ to  $x^+_{c,i,\delta}$ by a straight line $\zeta$: $[0,\tau]\to M$, then $|\dot\zeta|\le\frac{\mu\delta}{\tau}$. Since $L=\frac 12\langle A^{-1}\dot x,\dot x\rangle+V(x)$ there exists some constant $a_1>0$ which is independent of $\delta$ such that $L(\zeta,\dot\zeta)\le a_1((\frac{\mu\delta}{\tau})^2+\delta^2)$. Let $\tau=\mu$, then one has
$$
A[\zeta|_{[0,\mu]}]\le 2a_1\mu\delta^2.
$$
If $\mu<\frac{\lambda_1^2}{3a_1}$, one would have $A[\zeta|_{[0,\mu]}]<A[\gamma_{c,i}|_{[0,T_i]}]$ for sufficiently large $i$. It is absurd since the curve $\gamma_{c,i}$ is assumed to be semi-static.
\end{proof}

\begin{lem}\label{lem3.4}
Assume the Ma\~n\'e set covers a neighborhood of the hyperbolic fixed point. Let $\gamma^+_c$ and $\gamma^-_c$ be $c$-semi static curve passing through $I^+_{c,\delta}$ and $I^-_{c,\delta}$ respectively, then
\begin{enumerate}
   \item if they both approach the origin in the direction of $\Lambda_{1,x}$ as $t\to\pm\infty$, then they approach it in opposite direction, i.e.
       $$
       \lim_{t\to\infty}\frac{\dot\gamma^+_c(t)}{\|\dot\gamma^+_c(t)\|}=\lim_{t\to-\infty}\frac{\dot\gamma^-_c(t)} {\|\dot\gamma^-_c(t)\|};
       $$
   \item either $\gamma^+_c$ or $\gamma^-_c$ approaches the origin in the direction of $\Lambda_{2,x}$.
\end{enumerate}
\end{lem}
\begin{proof}
If both curves $\gamma^+_c$ and $\gamma^-_c$ approach the origin in the direction of $\Lambda_{1,x}$ such that
$$
\lim_{t\to\infty}\frac{\dot\gamma^+_c(t)}{\|\dot\gamma^+_c(t)\|}=-\lim_{t\to-\infty}\frac{\dot\gamma^-_c(t)}{\| \dot\gamma^-_c(t)\|},
$$
there would be a sharp wedge-shaped region in $B_{\delta}(0)$, bounded by $\gamma^+_c$ and $\gamma^-_c$, with a vertex at the origin and denoted $\mathbb{W}$. None of semi static curves passes through $\mathbb{W}$ to approach the origin. The rest of the proof is applying Lemma \ref{lem3.3}.
\end{proof}
\begin{cor}\label{cor3.1}
Assume the Ma\~n\'e set covers a neighborhood of the hyperbolic fixed point. If neither $I^-_{c,\delta}$ nor $I^+_{c,\delta}$ contains point through which the semi-static approaches the origin in the direction of $\Lambda_{2,x}$, then, both sets $I^-_{c,\delta}$ and $I^+_{c,\delta}$ are connected.
\end{cor}
\begin{proof}
Be aware that the Ma\~n\'e set is assumed to cover the whole torus and assume that both $I^-_{c,\delta}$ and $I^+_{c,\delta}$ do not contain a point through which the semi-static approaches the origin in the direction of $\Lambda_{2,x}$. In this case, all semi-static curves passing through $I^+_{c,\delta}$ approaches the origin in the same direction. Otherwise, there would be semi-static curves $\gamma^{+}_c,\gamma^{-}_c$ passing through $I^{+}_{c,\delta},I^{+}_{c,\delta}$ respectively such that they approaches the origin in the same direction, but it contradicts Lemma \ref{lem3.4}. Therefore, there would be a sharp wedge-shaped region $\mathbb{W}$ in $B_{\delta}(0)$, bounded by $\gamma^+_c$ and $\gamma'^+_c$ with a vertex at the origin such that through each point in the interior of $\mathbb{W}$, there is a semi-static curve which does not approach the origin as $t\to\infty$. As these semi-static curves do not intersect each other, both $\gamma^+_c$ and $\gamma'^+_c$ approach the origin as $t\to\infty$, there must be a semi static curve which passes through $\mathbb{W}$ and approaches the origin as $t\to-\infty$, see the right of Figure \ref{fig8}. This contradicts Lemma \ref{lem3.4} again.
\end{proof}

In the following, we study the case that both $I^-_{c,\delta}$ and $I^+_{c,\delta}$ are connected. Passing through a point $x^+_{c,i,\delta}\in\partial B_{\delta}(0)\backslash I^+_{c,\delta}$, close to $I^+_{c,\delta}$, there is a unique $c$-semi static curve $\gamma_{c,i}$ which will get close to the origin and eventually depart from the disk $B_{\delta}(0)$. Let $x^+_{c,i,\delta},x'^+_{c,i,\delta}\in\partial B_{\delta}(0)$ be two sequence of points such that they approach $I^+_{c,\delta}$ from different sides, let $x^+_{c,\delta},x'^+_{c,\delta}$ be the endpoint of $I^+_{c,\delta}$ such that $x^+_{c,i,\delta}\to x^+_{c,\delta}$ and $x'^+_{c,i,\delta}\to x'^+_{c,\delta}$. $I^+_{c,\delta}$ is a point if $x^+_{c,\delta}=x'^+_{c,\delta}$. Let $\gamma_{c,i}$ and $\gamma'_{c,i}$ be the semi static curves passing through the points $x^+_{c,i,\delta}$ and $x'^+_{c,i,\delta}$ respectively, they shall intersect the circle $\partial B_{\delta}(0)$ at points $x^-_{c,i,\delta}$ and $x'^-_{c,i,\delta}$ when they are going to leave the disk. Some points $x^-_{c,\delta},x'^-_{c,\delta}\in\partial B_{\delta}(0)$ exist such that $x^-_{c,i,\delta}\to x^-_{c,\delta}$ and $x'^-_{c,i,\delta}\to x'^-_{c,\delta}$ as $i\to\infty$.

If $x^-_{c,\delta}\ne x'^-_{c,\delta}$, they divide the circle $\partial B_{\delta}(0)$ into two arcs. The arc which does not contain $I^+_{c,\delta}$ is nothing else but $I^-_{c,\delta}$. Indeed, by time translation, certain $T^{\pm}_i>0$ exist such that $\gamma_{c,i}(-T^+_i)=x^+_{c,i,\delta}$, $\gamma_{c,i}(T^-_i)=x^-_{c,i,\delta}$ and $d(\gamma_{c,i}(0),0)=\min_td(\gamma_{c,i}(t),0)$. As $x^-_{c,i,\delta}\to x^-_{c,\delta}$, one has $T^-_i\to\infty$. It follows that both $x^-_{c,\delta}$ and $x'_{c,\delta}$ are in $I^-_{c,\delta}$. Since it is connected, that arc is exactly the set $I^-_{c,\delta}$ itself. If $x^-_{c,\delta}=x'_{c,\delta}$, $I^-_{c,\delta}$ is just a point.

Recall four points $(x^+_{\delta},x'^+_{\delta},x^-_{\delta},x'^-_{\delta})$ defined before. There are four orbits $(\gamma^{\pm},\dot\gamma^{\pm})$,  $(\gamma'^{\pm},\dot\gamma'^{\pm})$ such that the formulae (\ref{direction2}) (\ref{direction2'}) hold and $\gamma^{\pm},\gamma'^{\pm}$ intersect the circle $\partial B_{\delta}(0)$ at these four points. Let us consider the location of $x^+_{c,\delta}$ and $x'^+_{c,\delta}$ with respect to the points $x^+_{\delta}$ and $x'^+_{\delta}$. All possibilities are listed below:
\begin{enumerate}
   \item the two-point set $\{x^+_{c,\delta},x'^+_{c,\delta}\}$ does not have overlap with the set $\{x^+_{\delta},x'^+_{\delta}\}$;
   \item the set $\{x^+_{c,\delta},x'^+_{c,\delta}\}$ overlaps the set $\{x^+_{\delta},x'^+_{\delta}\}$.
\end{enumerate}
The following property is guaranteed by the condition $\lambda_1<\lambda_2$: by shrinking the size of $\delta$, the direction of $x^+_{c,\delta}$ ($x'^+_{c,\delta}$) can be set not close to the direction of $\Lambda_{2,x}$ if it is neither the point $x^+_{\delta}$  nor the point $x'^+_{\delta}$ itself.

For Case 1, because of Lemma \ref{lem3.4}, the semi static curves passing through $x^+_{c,\delta}$ and passing through $x'^+_{c,\delta}$ approach the origin in the same direction. By applying Lemma \ref{lemma3.1}, one sees that, when $x^+_{c,i,\delta}$ approaches $x^+_{c,\delta}$ from one side of $I^+_{c,\delta}$, the semi-static curve $\gamma_{c,i}$ will depart from the disk $B_{\delta}(0)$ in a direction close to $\Lambda_{2,x}$; the semi-static curve $\gamma'_{c,i}$ will depart from the disk $B_{\delta}(0)$ in a direction close to $-\Lambda_{2,x}$ when $x'^+_{c,i,\delta}$ approaches $x'^+_{c,\delta}$ from other side of $I^+_{c,\delta}$. So, the direction of $x^-_{c,i,\delta}$ is nearly opposite to the direction of $x'^-_{c,i,\delta}$ if $i\to\infty$. In this case, we claim that $\{x^-_{c,\delta},x'^-_{c,\delta}\}=\{x^-_{\delta},x'^-_{\delta}\}$. To see it, let us shrink the radius $\delta$ to a smaller $\delta'$. The curve $\gamma_{c,i}$ will intersect the circle $\partial B_{\delta'}(0)$ at a point $x^-_{c,i,\delta'}$ when it is going to leave the disk $B_{\delta'}(0)$. If $x^-_{c,\delta}\notin\{x^-_{\delta},x'^-_{\delta}\}$, the direction of $x^-_{c,i,\delta'}$would not be close to the direction of $\pm\Lambda_{2,x}$ provided $\delta'$ is sufficiently small and $i$ is sufficiently large. Along the curve we retreat from the point $x^-_{c,i,\delta'}$ to the point $x^+_{c,i,\delta'}$ where the curve enters the disk, it is guaranteed by Lemma \ref{lemma3.1} that the direction of $x^+_{c,i,\delta'}$, and consequently, the direction of $x^+_{c,\delta}$ is close to that of $\pm\Lambda_{2,x}$. It is obviously absurd. For the sequence $\{x'^-_{c,i,\delta}\}$ one has the same conclusion.
\begin{figure}[htp]
  \centering
  \includegraphics[width=9.5cm,height=4.2cm]{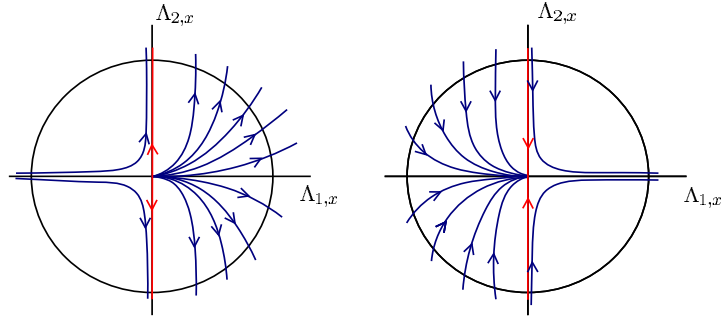}
  \caption{Two red lines in the left figure are the semi-static curves departing from the origin along the direction of $\Lambda_{2,x}$, the red lines in the right figure are the semi-static curves approaching the origin along the direction of $\Lambda_{2,x}$. }
  \label{fig9}
\end{figure}
Therefore, the set $I^-_{c,\delta}$ occupies an arc with length close to $\pi\delta$. Because the fixed point is hyperbolic, semi static curves crossing the arc $I^-_{c,\delta}$ produces orbits staying on the unstable manifold. Therefore, some number $d>0$ exists, no matter how small the number $\delta>0$ could be, such that the circle $\partial B_d(0)$ conatins an arc $I^-_{c,d}$ with length close to $\pi d$. Through each point of $I^-_{c,d}$ there is a semi static curve which approaches the prigin as the time approaches minus infinity, see the left figure in Figure \ref{fig9}. The right figure corresponds to case that the two-point set $\{x^-_{c,\delta},x'^-_{c,\delta}\}$ does not have overlap with the set $\{x^-_{\delta},x'^-_{\delta}\}$. In both cases, there are $c$-semi static curves which approaches the origin in the direction of $\Lambda_{2,x}$ as $t\to -\infty$.

If the set $\{x^+_{c,\delta},x'^+_{c,\delta}\}$ overlaps the set $\{x^+_{\delta},x'^+_{\delta}\}$, there exists at least one $c$-semi static curve which approaches the origin in the direction of $\Lambda_{2,x}$ as $t\to\infty$. With Corollary \ref{cor3.1}, it completes the proof of Theorem \ref{th3.1}.
\end{proof}

Since there are at most four orbits $(x_i(t),y_i(t))$ which approach the origin such that $x_i(t)$ approaches the origin in the direction of $\Lambda_{2,x}$ as the time $t\to\pm\infty$. Each of these orbits connects the origin to at most one Mather set. These Mather sets correspond to at most four edges (vertices) of $\mathbb{F}_0$.

\subsection{The Ma\~n\'e set for $\partial^*\mathbb{F}_0$: rational case and for $c\in\partial\mathbb{F}_0\backslash\partial^* \mathbb{F}_0$}
Let us consider $c\in\partial^*\mathbb{F}_0$ to see what will happen when the rotation vector of $\mu_c$ is rational and the strip $\mathcal{S}$ is filled with semi-static curves. In this case, both curves $\bar\pi\xi_c$ and $\bar\pi\xi'_c$ are closed circle which is associated with a homological class $g\in H_1(\mathbb{T}^2,\mathbb{Z})$. So we use $\xi_g,\xi'_g$ to denote the curves, i.e. $\xi_g=\xi_c$, $\xi'_g=\xi'_c$. In typical case, $\bar\pi\xi_g=\bar\pi\xi'_g$ and there exists an edge $\mathbb{E}_g\subset\partial\mathbb{F}_0$ such that, for each $c\in\mathbb{E}_g$, the Mather set $\mathcal{M}(c)$ is made up by the fixed point and the curve $\bar\pi\xi_g$, see Figure \ref{Fig1}. For any $c,c'\in\mathbb{E}_g$, one has
$$
\langle c'-c,g\rangle=0.
$$
The size of $\mathbb{E}_g$ is determined by the difference of actions along two minimal homoclinic orbits to the circle with opposite direction. Indeed, let $g^{\perp}\in H_1(\mathbb{T}^2,\mathbb{Z})/\mathrm{span}\{g\}$ be a irreducible class (generator), we consider the quantities
$$
A(g^{\perp})=\lim_{T\to\infty}\min_{\stackrel {\bar\pi\xi(-T)=\bar\pi\xi(T)}{\scriptscriptstyle [\xi]=g^{\perp}}} \int_{-T}^TL(\xi(t),\dot\xi(t))dt-\langle c,g\rangle
$$
where $[\xi]$ denotes the homological class of $\bar\pi\xi$ in the relative homological group and
$$
A(-g^{\perp})=\lim_{T\to\infty}\min_{\stackrel {\bar\pi\xi(-T)=\bar\pi\xi(T)}{\scriptscriptstyle [\xi]=-g^{\perp}}} \int_{-T}^TL(\xi(t),\dot\xi(t))dt-\langle c,g\rangle.
$$
Clearly, one has $A(g^{\perp})+A(-g^{\perp})\ge 0$ and it is generic that $A(g^{\perp})+A(-g^{\perp})> 0$ holds simultaneously for all $g\in H_1(\mathbb{T}^2,\mathbb{Z})$ because there are countably many homological classes only. In this case the size of $\mathbb{E}_g$ is given by
$$
|\mathbb{E}_g|=\max_{c,c'\in\mathbb{E}_g}|c-c'|=\frac 1{2\pi}(A(g^{\perp})+A(-g^{\perp})).
$$
\begin{pro}\label{pro3.2}
If $|\mathbb{E}_g|>0$, then for any $c\in\mathrm{int}\mathbb{E}_g$ the Ma\~n\'e set does not cover the whole configuration space $\mathbb{T}^2$.
\end{pro}
\begin{proof}
We will make use of the following lemma which can be proved with the same argument to prove Lemma \ref{lem3.2}.
\begin{lem}\label{lem3.5}
If $c\in\partial^*\mathbb{F}_0$ such that $\rho(\mu_c)$ is rational, $\bar\pi\xi_g=\bar\pi\xi'_g$ and if the configuration space $\mathbb{T}^2$ is covered by Ma\~n\'e set $\mathcal{N}(c)$ for certain $c\in\mathbb{E}_g$, then, passing through each point on the torus there is exactly one $c$-semi static curve.
\end{lem}
By the condition, the Aubry set for $c$ is the same as the Ma\~n\'e set and covers the whole configuration space. Indeed, given a small $\delta>0$, both $I^-_{c,\delta}$ and $I^+_{c,\delta}$ are non-empty and closed. As $I^-_{c,\delta}\cap I^+_{c,\delta}=\varnothing$, the fixed point is on the boundary of the set made up by those semi static curves which do not approach the fixed point. Such curve is obviously static since its $\alpha$-limit set is the same as its $\omega$-limit set. By a result of \cite{Mas}, all cohomology classes in the interior of a flat share the same Aubry set, and $\mathcal{N}(c')\supsetneq\mathcal{A}(c)$ holds for $c\in\mathrm{int}\mathbb{E}_g$ and $c'\in\partial\mathbb{E}_g$. Therefore, it is impossible that $\mathcal{N}(c)$ covers the whole configuration space for $c\in\mathrm{int}\mathbb{E}_g$.
\end{proof}

Because of Lemma \ref{lem3.5}, by completely following the argument to prove Theorem \ref{th3.1} one obtains

\begin{theo}\label{pro3.3}
For Hamiltonian of $($\ref{averagesystem}$)$ we assume typical potential $V$ so that $\bar\pi\xi_g=\bar\pi\xi'_g$ and the hypothesis $(${\bf H1}$)$ holds. If $\exists$ $c\in\mathbb{E}_g$ such that $\mathcal{N}(c)=\mathbb{T}^2$, then there exists $c$-semi static curve which approaches the fixed point in the direction of $\Lambda_{2,x}$ as $t\to\infty$ or $t\to-\infty$.
\end{theo}

Finally, let us study the case when $c\in\partial\mathbb{F}_0\backslash\partial^*\mathbb{F}_0$ which contains at most countably many vertices. Indeed, if both $\partial\mathbb{F}_0\backslash\partial^*\mathbb{F}_0$ and $\partial^*\mathbb{F}_0$ are non-empty, there do exist countably many vertices (cf. Theorem \ref{flatthm3}).

Let $\mathbb{E}_i\subset\partial\mathbb{F}_0\backslash\partial^*\mathbb{F}_0$ be an edge joined to other two edges at the vertex $c_i$, $c_{i+1}$ respectively. By Theorem \ref{flatthm3}, the Aubry set for $c_j$ consists of two minimal homoclinic curves $\gamma_{j-1}$ and $\gamma_j$. Denote by $g_j\in\mathbb{Z}^2$ the homology class of $\gamma_j$, then the matrix $(g_{j-1},g_j)$ is uni-module. By introducing suitable coordinates on $\mathbb{T}^2$, we can assume $g_i=(1,0)$. In this coordinate system, $g_{i-1}=(k,1)$ and $g_{i+1}=(k',-1)$.
\begin{figure}[htp] 
  \centering
  \includegraphics[width=7.0cm,height=2.5cm]{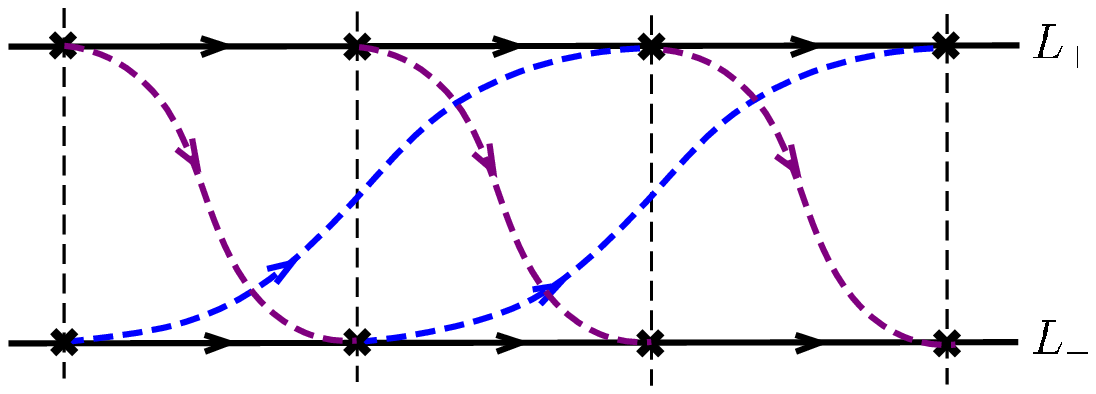}
\end{figure}
In this figure, each unit square represents a fundamental domain of $\mathbb{T}^2$ in the universal covering space, the horizontal line represents the lift of the homoclinic curve $\gamma_i$, which stays in the Aubry set for each $c\in\mathbb{E}_i$. The blue dashed lines represent the lift of the $\gamma_{i-1}$ which stays in the Aubry set for the class at one end-point of $\mathbb{E}_i$. The purple dashed lines represents the lift of the $\gamma_{i+1}$ which stays in the Aubry set for the class at another end-point of $\mathbb{E}_i$.

For any $c\in\mathbb{E}_i$, each $c$-semi static curve is static as the minimal measure is uniquely supported on the fixed point. In other words, the Aubry set is the same as the Ma\~n\'e set. So, if the Ma\~n\'e set for some $c\in\mathbb{E}_i$ covers the whole 2-torus, passing through each point on the torus there exists exactly one $c$-semi static curve. Therefore, in the same way to prove Proposition \ref{pro3.2} and Lamma \ref{lem3.2} we can prove  the following proposition

\begin{theo}\label{pro3.4}
For an edge $\mathbb{E}_i\subset\partial\mathbb{F}_0\backslash \partial^*\mathbb{F}_0$, it holds simultaneously for all $c\in\mathrm{int}\mathbb{E}_i$ that $\mathcal{N}(c)\subsetneq\mathbb{T}^2$. If for certain $c\in\partial\mathbb{E}_i$ such that $\mathcal{N}(c)=\mathbb{T}^2$, then there exists at least one static curve which approaches the origin in the direction of $\Lambda_{2,x}$.
\end{theo}

Because of Theorems \ref{th3.1}, \ref{pro3.3} and \ref{pro3.4}, there are at most four cohomological classes in $\partial\mathbb{F}_0$ for which it is possible that the Ma\~n\'e set covers the whole configuration space $\mathbb{T}^2$. One has

\begin{theo}\label{thm3.1}
For Hamiltonian $H=\frac 12\langle Ay,y\rangle-V(x)$, there exists a resudual set $\mathfrak{V}\subset C^r(\mathbb{T}^2,\mathbb{R})$ with $r\ge 2$ such that for each $V\in\mathfrak{V}$ it holds simultaneously for all $c\in\partial\mathbb{F}_0$ that the Ma\~n\'e set does not cover the whole configuration space: $\mathcal{N}(c)\subsetneq\mathbb{T}^2$.
\end{theo}
\begin{proof}
There are only four orbits $(\gamma^{\pm}(t),\dot\gamma^{\pm}(t)),(\gamma'^{\pm}(t),\dot\gamma'^{\pm}(t))$ which approach the hyperbolic fixed point in the direction of $\pm\Lambda_{2,x}$ as $t\to\infty$ or as $t\to -\infty$, i.e. Formulae (\ref{direction2}) (\ref{direction2'}) hold. Each of these orbits connects the fixed point to at most one Mather set, denoted by $\mathcal{M}^{\pm},\mathcal{M}'^{\pm}$. It may not connect to Mather set if it is not a minimal orbit. By the work of \cite{BC} it is generic property that, for each first cohomology class, the Mather set contains at most three connected components. Because the support of ergodic minimal measure is compact, each Mather sets under consideration contains the fixed point, there exists a small disk $U$ in a neighborhood of the fixed point such that $\bar U$ is disjoint with the set $\mathcal{M}^{\pm}\cup\mathcal{M}'^{\pm}$. The disk $U$ is chosen so that any curve of $\gamma^{\pm}$, $\gamma'^{\pm}$ does not hit $\bar U$ provided it is semi-static. Given $\delta>0$ the duration for a semi static curve staying outside of $\delta$-neighborhood of Mather set is finite.

We construct potential perturbation $V\to V+V_{\delta}$ such that non-negative $V_{\delta}$ is small in $C^r$-topology and $\mathrm{supp}V_{\delta}\subseteq\bar U$. Since $V_{\delta}$ is non-negative, a curve of of $\gamma^{\pm}$, $\gamma'^{\pm}$ is still semi static for certain $c$ if it is already semi-static before the perturbation is added. So, if a Ma\~n\'e set $\mathcal{N}(c)$ covers the $2$-torus, under the small perturbation constrcuted above, the set $U$ does not overlap with the Ma\~n\'e set. It is possible that another curve of $\gamma^{\pm}$, $\gamma'^{\pm}$ become semi static after the small perturbation is added. In this case, we can construct small perturbation further so that the support of potential perturbation does not touch semi static curves related to these Mather sets. As there are at most four Mather sets are concerned about, we can do it one by one. This completes the proof.
\end{proof}

\subsection{Annulus of cohomology equivalence around the disk $\mathbb{F}_0$}
To establish the new version of cohomology equivalence around the flat $\mathbb{F}_0$, we make use of the upper semi-continuity of Ma\~n\'e set with respect to the first cohomology class. According to Theorem \ref{thm3.1}, for typical potential $V$ in the Hamiltonian $G$ of (\ref{averagesystem}), it holds simultaneously for all $c\in\partial\mathbb{F}_0$ that the Ma\~n\'e set does not cover the configuration space: $\mathcal{N}(c)\subsetneq\mathbb{T}^2$.  As $\partial\mathbb{F}_0$ is compact, certain $\Delta_0>0$ exists so that for all $c\in\alpha_G^{-1}(\Delta)$ with $\Delta\le 2\Delta_0$, the Ma\~n\'e set does not cover the 2-torus too.

Since the Lagrangian is defined on $2$-torus,  for each average action $\Delta>\min\alpha_{G}$, the dynamics on the energy level $G^{-1}(\Delta)$ is similar to twist and area-preserving map. First of all, the rotation vector of each minimal measure is not zero. Thus, any minimal measure is not supported on fixed points. Next, for any class $c\in\alpha_{G}^{-1}(\Delta)$, all $c$-minimal measures share the same rotation rotation direction. Otherwise, the Lipschitz graph property of Mather set will be violated. Each ergodic minimal measure is supported on a periodic orbit if the rotation direction is rational.

From these properties one derives the following: for each $c\in\alpha^{-1}_{G}(\Delta)$, there exists a circle $S_c\subset\mathbb{T}^2$ such that each semi-static curve passes through it transversally and $\mathcal{N}(c)\cap S_c\subsetneq S_c$. Since Ma\~n\'e set is closed, there exist finitely many intervals $I_{c,i}\subset S_c$ disjoint to each other such that $\mathcal{N}(c)\cap S_c\subset\cup I_{c,i}$, see the left figure in Figure \ref{fig10}. In the extended configuration space $\mathbb{T}^3$, we choose a section $\Sigma_c=S_c\times\mathbb{T}$. As $G$ is independent of $\tau$, the Ma\~n\'e set in the extended space, denoted by $\mathcal{N}_G(c)$, stays in the strips: $\mathcal{N}_G(c)\cap\Sigma_c\subset (\cup I_{c,i})\times\mathbb{T}$, see the figure in middle of Figure \ref{fig10}.

Recall the Hamiltonian of (\ref{averagesystem}) is a truncation of the Hamiltonian of (\ref{mainsystems}) rewritten in the following
\begin{equation*}
G_{\epsilon}=\frac 12\langle Ay,y\rangle -V(x)+\sqrt{\epsilon}R_{\epsilon}(x,y,\tau),\qquad (x,y)\in\mathbb{T}^2\times\mathbb{R}^2
\end{equation*}
where $\tau=\frac{\sqrt{\epsilon}}{\omega_3} x_3$. The function $G_{\epsilon}$ solves the equation $\tilde G_{\epsilon}(x,\frac{\omega_3}{\sqrt{\epsilon}}\tau,y,-\frac{\sqrt{\epsilon}}{\omega_3} G_{\epsilon})=0$ where the Hamiltonian $\tilde G_{\epsilon}$ is a normal form of $H$ of (\ref{Eq1}) around double resonant point, given by (\ref{Hamiltonian}). Let $\mathcal{N}_{G_{\epsilon}}(c)$ denote the Ma\~n\'e set for $G_{\epsilon}$ in the extended configuration space.
Let $\mathbb{F}_{0,G_{\epsilon}}=\{c:\alpha_{G_{\epsilon}}=(c)=\min\alpha_{G_{\epsilon}}\}$. Because of the upper-semi continuity of Ma\~n\'e set with respect to small perturbation of Lagrangian (Hamiltonian), given $c\in\partial \mathbb{F}_{0,G_{\epsilon}}$ there exists $\epsilon_c>0$ so that for each $\epsilon\le\epsilon_c$ one has
\begin{equation}\label{c-equivalence}
\mathcal{N}_{G_{\epsilon}}(c)\cap\Sigma_c\subset (\cup I_{c,i})\times\mathbb{T},
\end{equation}
see the right figure of Figure \ref{fig10}.
\begin{figure}[htp]
  \centering
  \includegraphics[width=10cm,height=3cm]{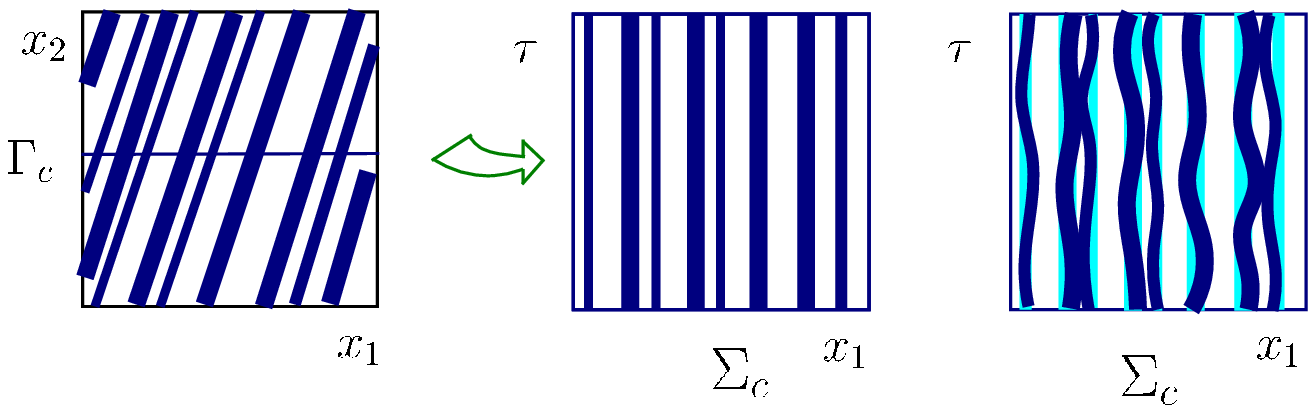}
  \caption{}
  \label{fig10}
\end{figure}
As the boundary of $\mathbb{F}_{0,G_{\epsilon}}$ is compact, some $\epsilon_0>0$ exists such that the relation (\ref{c-equivalence}) holds simultaneously for all $c\in\partial\mathbb{F}_{0,G_{\epsilon}}$ and $\epsilon\le\epsilon_0$. Using the upper semi-continuity of Ma\~n\' e set with respect to small perturbation of Lagrangian again, one obtains that Formula (\ref{c-equivalence}) holds for those $c$ such that $\alpha_{G_{\epsilon}} (c)\in (\min\alpha_{G_{\epsilon}},\min\alpha_{G_{\epsilon}}+\Delta_0)$.

To establish the cohomology equivalence, let us study what is the relation between the $\alpha$-function of the original autonomous Hamiltonian and the reduced Hamiltonian when it is restricted on certain energy level set.
\begin{theo}\label{flatthm5}
For the Hamiltonian $H(x,x_n,y,y_n)$ we assume that $\partial_{y_n}H\neq 0$ on $\{H^{-1}(E)\}$. Let $y_n=-\lambda Y(x,y,\tau)$ be the solution of $H(x,\frac 1{\lambda}\tau,y,-\lambda Y)=E$, let $\alpha_H$ and $\alpha_G$ be the $\alpha$-function for $L_H$ and $L_G$ respectively, where
$$
L_H(x,x_n,\dot x,\dot x_n)=\max_{y,y_n}\langle (\dot x,\dot x_n),(y,y_n)\rangle -H(x,y,x_n,y_n),
$$
$$
L_Y(x,y,\tau)=\max_{y}\langle \dot x,y\rangle -Y(x,y,\tau),
$$
Then we have $(c,-\lambda \alpha_Y(c))\in \alpha_H^{-1}(E)$.
\end{theo}
\begin{proof}
Let $\tilde c=(c,-\lambda\alpha_Y(c))$, $\tilde\gamma=(\gamma,\gamma_n)$, $\tilde x=(x,x_n)$ and $\tilde y=(y,y_n)$. Let $\gamma$ be $c$-minimal curve for the Lagrange flow $\phi_{L_F}^t$,  $\tilde\gamma$ is then $\tilde c$-minimal curve for the Lagrange flow $\phi_{L_H}^t$ if $\gamma_n=x_n$ and $\tilde\gamma$ is re-parameterized $\tau\to t$. If $x=x(\tau)$ is a solution of $\phi_{L_F}^t$, one obtains $y=y(\tau)$ from the Hamiltonian equations. Since $H(\tilde x(t),\tilde y(t))\equiv E$, we find
\begin{align*}
[A_Y(\gamma)]&=\int\Big(\Big\langle\frac{dx}{d\tau},y-c\Big\rangle-Y+\alpha_Y(c)\Big)d\tau\\
&=\int(\langle\dot{\tilde x},\tilde y-\tilde c\rangle -H+E)dt\\
&=[A_H(\tilde\gamma)].
\end{align*}
This completes the proof.
\end{proof}
Let $\pi_3: \mathbb{R}^3\to\mathbb{R}^{n-1}$ be the projection $\pi_3\tilde x=x$. By this theorem, $\pi_3^{-1}:H^1(\mathbb{T}^{2},\mathbb{R})\to \alpha_H^{-1}(E)$ is a homeomorphism for $c\in\mathbb{F}_0+d$, the $d$-neighborhood of the flat $\mathbb{F}_0$. Thus, what we obtained in this subsection have their counterpart in the energy level set $\{H^{-1}(E)\}$ where the class $\tilde c\in\pi_3^{-1}(\mathbb{F}_0+d)\cap\alpha^{-1}_H(E)$.   By such observation we obtain a cohomology equivalence as follows
\begin{theo}
Let $\alpha_{\tilde G_{\epsilon}}$ be the $\alpha$-function for the Hamiltonian $\tilde G_{\epsilon}$. The annulus-like surface
$$
\tilde{\mathbb{A}}=\alpha_{\tilde G_{\epsilon}}^{-1}(0)\cap\Big\{c_3\in \Big(-\frac{\sqrt{\epsilon}}{\omega_3}(\min\alpha_{G_{\epsilon}}+\Delta_0),\frac{\sqrt{\epsilon}}{\omega_3} \min\alpha_{G_{\epsilon}}\Big)\Big\}
$$
admits a foliation of contour circles of $\Gamma_{\lambda}=\{\tilde c:\alpha_{\tilde G_{\epsilon}}(\tilde c)=0,c_3=\lambda\}$, each of these circles establishes a relation of cohomology equivalence.
\end{theo}
\begin{proof}
As the Hamiltonian $G_{\epsilon}$ is a solution of the equation $\tilde G_{\epsilon}(x,\frac{\omega_3}{\sqrt{\epsilon}}\tau,y,-\frac{\sqrt{\epsilon}}{\omega_3} G_{\epsilon})=0$, each orbit of the Hamiltonian flow $\Phi^{\tau}_{G_{\epsilon}}$ in the extended phase space $\mathbb{R}^2\times\mathbb{T}^3$ is the projection of certain orbit of $\Phi^{\tau}_{\tilde G_{\epsilon}}$ in the energy level set $\tilde G_{\epsilon}^{-1}(0)$:
$$
(y(t),y_3(t),x(t),x_3(t))\to \Big(y(\tau),x(\tau),\frac{\omega_3}{\sqrt{\epsilon}}\tau\Big),
$$
where $\tau$ is parameterized by $t$ so that $\tau(t)=\frac{\sqrt{\epsilon}}{\omega_3} x_3(t)$. It is guaranteed by the condition $\partial_{x_3}\tilde G_{\epsilon}\ne 0$. Therefore, for autonomous Hamiltonian $\tilde G_{\epsilon}$, the Ma\~n\'e set $\mathcal{N}_{\tilde G_{\epsilon}}(c,c_3)$ is the same as the Ma\~n\'e set $\mathcal{N}_{G_{\epsilon}}(c)$ in the extended configuration space if $c_3=-\frac{\sqrt{\epsilon}}{\omega_3}\alpha_{G_{\epsilon}}(c)$.
So, if the section $\Sigma_c$ is chosen such that $\Sigma_c=S_c\times\mathbb{T}$, we obtain from (\ref{c-equivalence}) that
$$
V_c=\mathrm{span}\{(0,0,1)\}, \qquad \forall\ c\in \alpha_{G_{\epsilon}}^{-1} (\min\alpha_{G_{\epsilon}},\min\alpha_{G_{\epsilon}}+\Delta_0).
$$
In this case, one has $V_c^{\perp}=\mathrm{span}\{(1,0,0),(0,1,0)\}$. For any two classes $\tilde c$, $\tilde c'$ on the same circle $\Gamma_{\lambda}$, one has $\tilde c-\tilde c'=(c-c,0)\in V^{\perp}_c$, namely, if two classes are located on the same circle, then they are cohomologically equivalent.
\end{proof}

Let $\mathbb{A}=\pi_3\tilde{\mathbb{A}}$, it is an annulus in $\mathbb{R}^2$ sorrounding the flat $\mathbb{F}_{0,G_{\epsilon}}$. By the definition one can see
$$
\mathbb{A}=\{c:0<\alpha_{G_{\epsilon}}(c)-\min\alpha_{G_{\epsilon}}\le\Delta_0\}.
$$
As $G_{\epsilon}$ is a small perturbation of the mechanical system $G$, the thickness of the annulus $\mathbb{A}$ is of order $\sqrt{\Delta_0}$ for small $\Delta_0$. It is independent of $\epsilon$.

Recall the definition of Legendre-Fenchel duality between the first homology and the first cohomology. Given a Tonelli Hamiltonian $H$, one obtains the $\alpha$- as well as the $\beta$-function, denoted by $\alpha_H$ and $\beta_H$ respectively, which in turn defines $\mathscr{L}_{\beta_H}$: a first cohomology class $c\in \mathscr{L}_{\beta_H}(g)$  if $\alpha_H(c)+\beta_H(g)=\langle c,g\rangle$.

Let us recall Theorem 1.1 of \cite{C15}. Given two irreducible classes $g,g'\in H_1(\mathbb{T}^2,\mathbb{Z})$, for generic potential $V$ in the Hamiltonian $\tilde G_{\epsilon}$ of (\ref{mainsystems}), there exist two channels $\mathbb{W}_g$, $\mathbb{W}_{g'}\subset H^1(\mathbb{T}^2,\mathbb{R})$
$$
\mathbb{W}_g=\bigcup_{\nu\ge\nu_g}\mathscr{L}_{\beta_{G_{\epsilon}}}(\nu g), \qquad \mathbb{W}_{g'}=\bigcup_{\nu\ge\nu_{g'}}\mathscr{L}_{\beta_{G_{\epsilon}}}(\nu g')
$$
where $\nu_g,\nu_{g'}>0$ such that $\alpha_{G_{\epsilon}}(\nu_gg), \alpha_{G_{\epsilon}}(\nu_{g'}g')=3\epsilon^d$.  It implies that both channels $\mathbb{W}_g$ and $\mathbb{W}_{g'}$ reach to $O(\epsilon^{\frac d2})$-neighborhood of the flat $\mathbb{F}_{0,G_{\epsilon}}$. Since the thickness of the annulus $\mathbb{A}$ is of order $\sqrt{\Delta_0}>0$ which is indeoendent of $\epsilon$, we obtain

\begin{theo}[Overlap Property]\label{overlap}
Given any two irreducible $g,g'\in H_1(\mathbb{T}^2,\mathbb{Z})$, there exists a positive number $\epsilon_0= \epsilon_0(V,g,g')>0$ so that the channels intersect the annulus: $\mathbb{W}_g\cap\mathbb{A} \ne\varnothing$ and $\mathbb{W}_{g'}\cap\mathbb{A}\ne\varnothing$ provided $0<\epsilon\le\epsilon_0$. Moreover,
\begin{enumerate}
   \item the channel $\mathbb{W}_g$ is connected to the channel $\mathbb{W}_{g'}$ by circles $\Gamma_{\lambda}$ of cohomology equivalence;
   \item For each $c\in\mathrm{int}\mathbb{W}_g$ $(\mathrm{int}\mathbb{W}_{g'})$ the Aubry set is located in a normally hyperbolic invariant cylinder with the homology type of $g$ $(g')$.
\end{enumerate}
\end{theo}

By the definition, two channels $\tilde{\mathbb{W}}_g,\tilde{\mathbb{W}}_{g'}\subset \alpha_{\tilde G_{\epsilon}}^{-1}(0)$ are well-defined so that $\pi_3\tilde{\mathbb{W}}_g=\mathbb{W}_g$ and $\pi_3\tilde{\mathbb{W}}_{g'}=\mathbb{W}_{g'}$. Also, a flat $\tilde{\mathbb{F}}_{0,\tilde G_{\epsilon}}\subset \alpha_{\tilde G_{\epsilon}}^{-1}(0)$ exists so that $\pi_3\tilde{\mathbb{F}}_{0,\tilde G_{\epsilon}}=\mathbb{F}_{0,G_{\epsilon}}$.

Back to the original coordinates $(p,q)\in\mathbb{R}^3\times\mathbb{T}^3$. For the Hamiltonian (\ref{Eq1}) one has the $\alpha$-function, denoted by $\alpha_H$. For $E>\min h$, the set $\alpha_{H}^{-1}(E)$ is homeomorphic to a sphere. A strong double resonance correspond to a $2$-dimensional falt on this sphere, denoted by $\tilde{\mathbb{F}}_{g,g'}$, with the size of order $\sqrt{\epsilon}$ and surrounded by an annulus $\tilde{\mathbb{A}}$,  The width of $\tilde{\mathbb{A}}$ is of order $\sqrt{\Delta_0\epsilon}$ where $\Delta_0>0$ is a small number but independent of $\epsilon$. The annulus $\tilde{\mathbb{A}}$ admits a foliation circles of cohomology equivalence.  The resonance relation determined by $g$ as well as $g'$ determines channel $\tilde{\mathbb{W}}_g$ as well as $\tilde{\mathbb{W}}_{g'}$, both get $\sqrt{\epsilon}^{1+d}$-close to the flat $\tilde{\mathbb{F}}_{g,g'}$ so they are connected by circles of cohomology equivalence.

\section{Generalized transition chain}
\setcounter{equation}{0}
The concept of transition chain was proposed by Arnold in \cite{A66} for the construction of diffusion orbits. It was formulated in geometric language. The generalized transition chain formulated in our previous work \cite{CY1,CY2,LC} is a combination of Arnold's mechanism and the mechnism of cohomology equivalence. It is in variational language which requires less information about the geometric structure.

\subsection{Definition for autonomous Hamiltonian} In this case, the minimal point set of barrier function can never be totally disconnected. To see this, let us recall the definition of barrier function in \cite{M93}:
$$
B_c^*(\tilde x)=\min_{m,m'\in\mathcal{M}(c)}\{h_c^{\infty}(m,\tilde x)+h_c^{\infty}(\tilde x,m')-h_c^{\infty}(m,m')\}.
$$
Clearly, the minimal point set of $B_c^*$ is exactly the set $B_c^{*-1}(0)=\mathcal{N}(c)$. For $x\in \mathcal{N}(c)$, there is at least one semi-static curve passing through this point. The barrier function reaches its minimum along the whole curve. For simplicity of notation we say $B_c^{*-1}(0)$ is totally disconnected in an open set $U$ if there exists certain section $D$ such that the restriction of $B_c^{*-1}(0)$ on $D$, denoted by $B_c^{*-1}(0)|_D$, is totally disconnected in $U$.

To define transition chain,  we need to work on some finite covering of the configuration manifold. Let $\check\pi$: $\check{M}\to M$ be a finite covering of the configuration space and let $\mathcal{N}(c,\check{M})$ be the Ma\~n\'e set with respect to $\check M$. In some cases one has $\check\pi \mathcal{N}(c,\check{M})\supsetneq\mathcal{N}(c)$. For instance, if a Mather set is located on a lower dimensional torus which is normally hyperbolic, the $\alpha$-function will have a flat. For each cohomology class in the interior of the flat, the Ma\~n\'e set is the same as the Mather set.  If we choose a covering space so that the lift of the Mather set contains two connected components,  then the Ma\~n\'e set with respect to the covering space will contain the minimal orbit connecting these two components. A concret example is a pendulum $H(x,y)=\frac 12(y^2-(1-\cos x))$. The set $\check\pi\mathcal{N}(0,2\mathbb{T})$ is made up by the fixed point and the separatrix and $\mathcal{N}(0,\mathbb{T})$ is just the fixed point.

\begin{defi} [Autonomous Case]\label{chaindef1}
Let $c$, $c'$ be two cohomolgy classes in $H^1(M,\mathbb{R})$. We say that $c$ is joined with $c'$ by a generalized transition chain if a continuous curve $\Gamma$: $[0,1]\to H^1(M,\mathbb{R})$ exists such that $\Gamma(0)=c$, $\Gamma(1)=c'$, $\alpha(\Gamma(s))\equiv E>\min\alpha$ and for each $s\in [0,1]$  at least one of the following cases takes place:
\begin{enumerate}
   \item there exist certain finite covering $\check{\pi}:\check{M}\to M$, two open domains $N_1,N_2\subset\check M$ with $d(N_1,N_2)>0$, a codimeniosn one disk $D_{s}$ and small numbers $\delta_s,\delta'_s>0$ such that
   \begin{enumerate}
      \item $\mathcal{A}(\Gamma(s))\cap N_1\neq\varnothing$, $\mathcal{A}(\Gamma(s))\cap N_2\neq\varnothing$ and $\mathcal{A}(\Gamma(s'))\cap (N_1\cup N_2)\neq\varnothing$ for each $|s'-s|<\delta_s$,
      \item there exists a section $D_s$ such that $\check\pi\mathcal{N}(\Gamma(s),\check M)|_{D_{s}}\backslash (\mathcal{A}(\Gamma(s))+\delta'_s)$ is non-empty and totally disconnected;
   \end{enumerate}
   \item there exists $\delta_s>0$, for each $s'\in (s-\delta_s,s+\delta_s)$, $\Gamma(s')$ is equivalent to $\Gamma(s)$.
\end{enumerate}
\end{defi}

In Case 1, if the Aubry set contains only one Aubry class, one can take some finite covering $\check\pi:\check M\to M$ non-trivial if $H_1(M,\mathcal{A},\mathbb{Z})\neq 0$. A typical case is that $\mathcal{A}(\Gamma(s))$ is contained in a small neighborhood of lower dimensional torus. One takes suitable finite covering space so that $\mathcal{A}(\Gamma(s),\check M)$ contains exactly two connected components. If $\mathcal{A}(\Gamma(s))$ contains more than one class, we choose $\check M=M$.

By the definition, for each cohomology class $\Gamma(s)$, the Aubry set $\tilde{\mathcal{A}}(\Gamma(s))$ can be connected to certain Aubry set $\tilde{\mathcal{A}}(\Gamma(s'))$ nearby, either by Arnold's mechanism which looks like heteroclinic orbits as shown in the first case, or by cohomology equivalence. The existence of generalized transition chain implies the existence of sequence of local connecting orbits $(\gamma_i,\dot\gamma_i)$, a sequence of numbers $s_i$ such that $\alpha(\gamma_i,\dot\gamma_i)\subset \mathcal{A}(\Gamma(s_i))$ and $\omega(\gamma_i,\dot\gamma_i))\subset\mathcal {A}(\Gamma(s_{i+1}))$. Global connecting orbits are constructed shadowing these local connecting orbits, one can refer \cite{LC} for the details of construction of diffusion orbits for autonomous Hamiltonian.

\subsection{Diffusion path: generalized transition chain}
It has become very clear that a generalized transition has been found to cross double resonance, it is a  cusp-residual property.

Let $\mathfrak{B}_a\subset C^r(\mathbb{T}^3\times B)$ be a ball, $F\in\mathfrak{B}_a$ if and only $\|F\|_{C^r}\le a$. Let $\mathfrak{S}_a=\partial \mathfrak{B}_a$ be the sphere, $F\in\mathfrak{S}_a$ if and only $\|F\|_{C^r}=a$. It inherits the topology from $C^r(\mathbb{T}^3\times B)$.
Let $\mathfrak{R}_a$ be a set residual in $\mathfrak{S}_a$, each $P\in\mathfrak{R}_a$ is associated with a set $R_P$ residual in the interval $[0,a_P]$ with $a_P\le a$. A set $\mathfrak{C}_a$ is said cusp-residual in $\mathfrak{B}_a$ if
$$
\mathfrak{C}_a=\{\lambda P:P\in\mathfrak{R}_a,\lambda\in R_P\}.
$$
For the nearly integrable Hamiltonian of (\ref{Eq1}) we rewrite here
\begin{equation*}
H(p,q)=h(p)+\epsilon P(p,q),\qquad (p,q)\in\mathbb{R}^3\times\mathbb{T}^3,
\end{equation*}
we choose a ball $B\subset\mathbb{R}^3$ such that $h^{-1}(E)\subset B$.

Given a resonance path $\Gamma_{k'}=\{p\in h^{-1}(E):\langle\partial h(p),k'\rangle=0\}$, a typical perturbation $P$ determines certain $\epsilon_P>0$ such that $\forall$ $\epsilon\in(0,\epsilon_P)$ there are finitely many points $\{p''_i\in\Gamma_{k'}\}$ which have to be handled as strong double resonance, the number of such points is independent of $\epsilon$ (see Section 6 of \cite{C15}). In an $O(\epsilon^{\sigma})$-neighborhood of such points ($\sigma<\frac 16$) one obtains a normal form $\tilde G_{\epsilon}$ of (\ref{Hamiltonian}) rewritten here, in which $\tilde y$ ranges over the set $|\tilde y|=\epsilon^{\sigma-\frac 12}$,
\begin{equation}\label{normalformat2resonance}
\tilde G_{\epsilon}=\frac{\omega_3}{\sqrt{\epsilon}}y_3+\frac 12\langle\tilde A\tilde y, \tilde y\rangle
-V(x)+\sqrt{\epsilon}\tilde R_{\epsilon}(\tilde x,\tilde y).
\end{equation}
Because of a result in \cite{Lo}, the path $\Gamma_{k'}$ is covered by $O(T_i^{-1}\epsilon^{\sigma})$-neighborhood of double resonant points $p''_i$, where $T_i=\min\{T\in\mathbb{Z}^+:T\partial h(p''_i)\in\mathbb{Z}^3\backslash 0\}$ is the period of $\partial h(p''_i)$ and $T_i\le K\epsilon^{-\frac 13(1-3\sigma)}$. Restricted on $O(T_i^{-1}\epsilon^{\sigma})$-neighborhood of $p''_i$ one also obtains the normal form of (\ref{normalformat2resonance}) where the potential $V$ admits a decomposition
$$
V(x)=V_1(\langle k',\tilde x\rangle)+V_2(\langle k',\tilde x\rangle,\langle k''_i,\tilde x\rangle),
$$
where the irreducible integer $k''_i\in\mathbb{Z}^3$ is the second resonance relation.
For typical $P$, the second derivative of $V_1$ at its minimal point is a positive number independent of $\epsilon$. The $C^2$-norm of $V_2$ decreases to zero as $|k''_i|\to\infty$. Consequently, except for finitely many (independent of $\epsilon$), all other double resonance can be treated as single resonance, the Hamiltonian flow $\Phi^t_{\tilde G_{\epsilon}}$ admits several NHICs which is a small deformation of NHICs of the flow produced by
$$
\frac{\omega_3}{\sqrt{\epsilon}}y_3+\frac 12\langle\tilde A\tilde y, \tilde y\rangle-V_1(\langle k',\tilde x\rangle),
$$
Both terms $V_2$ and $\sqrt{\epsilon}\tilde R_{\epsilon}$ are thought as small perturbation.

Recall the process to get the normal form (\ref{normalformat2resonance}), one introduces a linear coordinate transformation $(p,q)\to (A^{-t}p,Aq)$ so that $\partial h(p''_i)=(0,0,\omega_3)$, a step of KAM iteration $\Phi_i$ and homogenization.
$$
\tilde G_{\epsilon}=\frac 1{\epsilon}\Phi_i^*H, \qquad \tilde y=\frac 1{\sqrt{\epsilon}}\Big(p-p''_i\Big), \qquad \tilde x=q, \qquad s=\sqrt{\epsilon}t.
$$
Therefore, we obtain a candidate of generalized transition chain:

{\it The Hamiltonian flow $\Phi_H^t$ admits invariant cylinders, normal hyperbolic in $H^{-1}(E)$ which contain the Mather sets for all classes $\tilde c\in\tilde{\mathbb{W}}_{k',\tilde E}\backslash\cup(\tilde{\mathbb{F}}_{i,\tilde E}+\epsilon^{\frac{1+d}2})$ where $\tilde{\mathbb{W}}_{k',\tilde E}$ is defined as in $($\ref{resonantpathinc}$)$
$$
\tilde{\mathbb{F}}_{i,\tilde E}=\Big\{\tilde c\in H^1(\mathbb{T}^3,\mathbb{R}):\alpha_H(\tilde c)=\tilde E,\langle\rho(\mu_{\tilde c}),k'\rangle=0,\langle\rho(\mu_{\tilde c}),k''_i\rangle=0\Big\},
$$
$\rho(\mu_{\tilde c})$ denotes the rotation vector of ergodic $\tilde c$-minimal measure $\mu_{\tilde c}$ and the index $i$ is taken only from those strong double resonant points. Nevertheless, the hyperbolicity becomes much weaker, multiplied by the factor $\sqrt{\epsilon}$.}

In conclusion, a cusp-residual set $\mathfrak{P}_0\subset\mathfrak{B}_a$ exists so that for each $\epsilon P\in\mathfrak{P}$ one has a candidate of generalized transition chain. This set is open. As it shown above, given a typical $P\in C^r(B\times\mathbb{T}^3)$, these NHICs exist for any $\epsilon\in(0,\epsilon_P)$. For $P'$ approaches $P$, $\epsilon_{P'}$ approaches $\epsilon_P$. In the establishment of cohomology equivalence around double resonance, one only needs to perturb the function $P(p,q)\to P(p,q)+V(q)$ for finitely many double resonant points. Let $\mathfrak{B}_d(\epsilon P)\subset C^r(B\times\mathbb{T}^3)$ denote a ball with radius $d$, centered at $\epsilon P$.

\begin{lem}
Each perturbation $\epsilon P\in\mathfrak{P}_0$ is associated with a number $d=d(\epsilon P)>0$ so that the ball $\mathfrak{B}_d(\epsilon P)$ contains a residual set $\mathfrak{R}_d(\epsilon P)$. For each $\epsilon'P'\in\mathfrak{R}_d(\epsilon P)$, the Hamiltonian $H=h+\epsilon'P'$ admits a generalized transition chain connecting any two classes $\tilde c,\tilde c'\in\tilde{\mathbb{W}}_{k',\tilde E}\backslash\cup(\tilde{\mathbb{F}}_{i,\tilde E}+\epsilon^{\frac{1+d}2})$.
\end{lem}
\begin{proof}
Since the set $\mathfrak{P}_0$ is open, for each $\epsilon P\in\mathfrak{P}_0$, there exists some $d=d(\epsilon P)>0$ such that $\mathfrak{B}_d(\epsilon P)\subset\mathfrak{P}_0$. For each $\epsilon'P'\in \mathfrak{B}_d(\epsilon P)$ there is a candidate of generalized transition chain for the Hamiltonian $h+\epsilon'P'$. We want to show the existence of a set $\mathfrak{R}_d(\epsilon P)$ residual in $\mathfrak{B}_d(\epsilon P)$ such that for each $\epsilon'P'\in \mathfrak{R}_d(\epsilon P)$, the minimal homiclinic orbits to invariant $2$-torus in these NHICs are totally disconnected.

We use the technique designed in \cite{CY1}. In the homogenized coordinate, the normal form takes the form of (\ref{normalformat2resonance}). Restricted on the energy level set, it is equivalent to the system (\ref{mainsystems}) which has two and half degrees of freedom. The time-periodic map defined by the Hamiltonian flow has NHICs $\{\Pi_i\}$ which are symplectic sub-manifolds. We consider invariant circles $\{\Gamma_{\nu}\}$ on $\Pi_i$. Each circle is the time-periodic map of some Aubry set. Because of the Lipschitz property of Aubry set and the smoothness of the cylinder, these circles are all Lipschitz curves. Each $\Pi_i$ can be thought as the image of $\Pi=\{(x_1,x_2,y_1,y_2):x_2=y_2=0,x_1\in\mathbb{T},y_1\in [0,1]\}$ under the map $\psi$: $\Pi\to\Pi_i$. This map induces a $2$-form $\psi^*\omega$ on $\Pi$
$$
\psi^*\omega=D\psi dx_1\wedge dy_1
$$
where $D\psi$ is the Jacobian of $\psi$. Since the second de Rham cohomology group of $\Pi$ is trivial, it follows from Moser's argument on the isotopy of symplectic forms \cite{Mo} that there exists a diffeomorphism $\psi_1$ on $\Pi$ such that
$$
(\psi\circ\psi_1)^*\omega=dx_1\wedge dy_1.
$$
Since $\Pi$ is invariant for the time-periodic map $\Phi_{G_{\epsilon}}$ and $\Phi_{G_{\epsilon}}^*\omega=\omega$, one has
$$
\Big((\psi\circ\psi_1)^{-1}\circ\Phi_{G_{\epsilon}}\circ(\psi\circ\psi_1)\Big)^* dx_1\wedge dy_1=dx_1\wedge dy_1
$$
i.e. $(\psi\circ\psi_1)^{-1}\circ\Phi_{G_{\epsilon}}\circ(\psi\circ\psi_1)$ preserves standard area. These invariant circles $\Gamma_{\nu}$ are pulled back to the standard cylinder, denoted by $\Gamma^*_{\nu}$ which are obviously Lipschitz. Therefore, these invariant circles is parameterized by the area bounded by this circle and a prescribed circle \cite{CY1}
\begin{equation}\label{holder}
\|\Gamma^*_{c(\sigma)}-\Gamma^*_{c(\sigma')}\|_{C^0}\le C^*\sqrt{|\sigma-\sigma'|}.
\end{equation}
By using the normal hyperbolicity, the pull-back of local stable and unstable manifolds are also parameterized by $\sigma$ with $\frac 12$-H\"older continuity. Since the map $\psi,\psi_0$ are smooth, in the homogenized coordinates, these stable and unstable manifold are also parameterized $\sigma$ with $\frac 12$-H\"older continuity, may with different H\"older coefficient. Since these cylinders do not touch the region where $G_{\epsilon}\le \epsilon^d$ with suitably small $d>0$, one obtains from Formula (2.20) in \cite{C15} that the coefficient is bounded by $C^*M\epsilon^{-2d\mu_6}$, where $\mu_6>0$ is small, both $M$ and $\mu_6$ are independent of $\epsilon$. Back to the $(p,q)$-coordinate, the coefficient is bounded by $C^*M\epsilon^{-2d\mu_6-\frac 12}$.

The rest of the proof is following the argument in Section 6 of \cite{CY1}. We emphasize that the large coefficient $C^*M\epsilon^{-2d\mu_6-\frac 12}$ does not damage the proof. By choosing $d>0$ sufficiently small, $\epsilon'\ge\frac 12\epsilon$ holds for all $\epsilon'P'\in \mathfrak{B}_d(\epsilon P)$.

To show the transversal intersection of the stable and unstable manifolds we only need to consider those cohomology classes for which the Mather is a $2$-torus. These classes are parameterized by the area $\sigma$ so that (\ref{holder}) holds. One Hamiltonian produces a set of barrier functions $H\to\{B^*_{c(\sigma)}\}$. Because $\sigma$ is defined on a Cantor set on line, the box dimension of $\{B^*_{c(\sigma)}\}$ is not larger than $2$ in $C^0$-topology, due to the H\"older exponent in (\ref{holder}) which equals $\frac 12$, no matter how large the coefficient is.

Under the perturbation of bump function as designed in \cite{CY1}, the set of barrier function undergoes a translation $\{B^*_{c(\sigma)}\}\to\{B^*_{c(\sigma)}+\delta G\}$ when they are restricted on a disk $D$ where the minimal homoclinic curves pass through.

We define a set $\mathfrak{Z}\subset C^0$. A function $U\in\mathfrak{Z}$ if the set $U^{-1}(\min U)\cap D$ is not totally disconnected. Intuitively, this set has infinite ``codimensions" in $C^0(\mathbb{T}^2)$. Since the box dimension of $\{B^*_{c(\sigma)}\}$ is finite, there are abundant perturbations of shit $\{B^*_{c(\sigma)}\}\to\{B^*_{c(\sigma)}+\delta G\}$ so that $\{B^*_{c(\sigma)}+\delta G\}\cap\mathfrak{Z}=\varnothing$. It implies that all minimal homoclinic orbits are totally disconnected for perturbed Hamiltonian. Obviously, the perturbation can be arbitrarily small.

The perturbation to the generating function $G$ can be achieved by perturbing the Hamiltonian $H\to H'=H+\delta H$. Let $\Phi'$ be the generating function $G+\delta G$, the symplectic diffeomorphism $\Psi=\Phi'\circ\Phi^{-1}$ is close to identity. Let $\rho(s)$ be a smooth function such that $\rho(0)=0$ and $\rho(1)=1$, let $\Phi'_s$ be the symplectic map produced by $G+\rho(s)\delta G$ and let $\Psi_s=\Phi'_s\circ\Phi^{-1}$. Obviously, $\Psi_s$ defines a symplectic isotopy between identity and $\Psi$. Therefore, there is a family of symplectic vector fields $X_s$: $T^*M\to TT^*M$ such that
$$
\frac d{ds}\Psi_s=X_s\circ\Psi_s.
$$
By the choice of perturbation, there is a simply connected and compact domain $D$ such that $\Psi_s|_{T^*M\backslash D}=\mathrm{id}$. It follows that there is a Hamiltonian $H_1(x,y,s)$ such that $dH_1(Y)=dy\wedge dx(X_s,Y)$ holds for any vector field $Y$. By using a Weinstein's result one sees that $dH_1$ is small. Since Hamiltomorphism is a subgroup of symplectic diffeomorphism, there is a Hamiltonian $H'$ close to $H$ such that $\Phi_{H_1}\circ\Phi_H=\Phi^t_{H'}|_{t=1}$. For details, one can refer to Section 6 of \cite{CY1}.
\end{proof}

As the set $\mathfrak{P}_0$ are covered by those small balls $\mathfrak{B}_d(\epsilon P)$, there exist a set $\mathfrak{P}'$ residual in $\mathfrak{P}_0$ such that for each $\epsilon P\in\mathfrak{P}'$ the Hamiltonian flow admits a generalized transition chain. It follows from Kuratowski-Ulam theorem that some cusp-residual set $\mathfrak{P}$ exists such that a generalized transition chain exists for each $\epsilon P\in\mathfrak{P}$.

\subsection{Genericity in the sense of Ma\~n\'e}
We are going to study of potential perturbation of integrable Hamiltonian $H=h(p)+\epsilon P(q)$. In this case, one also has a cusp-residual set $\mathfrak{P}_0\subset\mathfrak{B}_a\subset C^r(\mathbb{T}^3)$ such that for each $\epsilon P(q)\in\mathfrak{P}_0$  one has a candidate of generalized transition chain. This set is also open.

To obtain generalized transition chain, let us construct the potential perturbations. It has been shown in \cite{C15} that for each $\tilde c\in\mathrm{int}\tilde{\mathbb{W}}_g$ the Mather set stays on some normally hyperbolic invariant cylinder (NHIC), denoted by $\tilde\Pi_i$ which is diffeomorphic to $[p,p']\times\mathbb{T}^2$. By suitably choosing a section of $\tilde\Pi_i$, denoted by $\Pi_i$ which is diffeomorphic to $[p,p']\times\mathbb{T}$, one obtains a return map which is twist and area-preserving. By \cite{CY1}, all invariant curves on $\Pi_i$ can be parameterized by ``area'' element $\sigma$: $\mathbb{I}_i\to \Upsilon_{\sigma}$ so that
$$
|\Upsilon_{\sigma}-\Upsilon_{\sigma'}|\le C_1\sqrt{|\sigma-\sigma'|}.
$$
Emanating from in invariant curve $\Upsilon_{\sigma}$ the orbits make up an invariant torus $\tilde\Upsilon_{\sigma}\subset\tilde\Pi_i$. We choose a covering space $\pi:\bar M\to M$ so that the lift of $\tilde\Pi_i$ has two components, denoted by $\tilde\Pi_{i,\ell}$ and $\tilde\Pi_{i,r}$. Correspondingly, the lift of those invariant torus also has two components $\tilde\Upsilon_{\ell,\sigma}$ and $\tilde\Upsilon_{r,\sigma}$.

In the configuration space $\mathbb{T}^3$, we choose a 2-dimensional disk $D$ which transversally intersects the backward semi-static curves $\gamma_{\tilde x,\sigma_0}^-:(-\infty,0]\to\bar M$ with $\gamma_{\tilde x,\sigma_0}^-(0)=\tilde x\in D$. These curves approach $\tilde\Upsilon_{l,\sigma_0}$ as $t\to -\infty$. In suitable coordinate system we can assume that $D$ is located in the section
$$
D+d_1=\{(x_1,x_2,x_3):x_1=x_{10},|x_2-x_{20}|\le d+d_1, |x_3-x_{30}|\le d+d_1\}.
$$
Let $D=(D+d_1)|_{d_1=0}$. We write the curve $\gamma^-_{\tilde x_0,\sigma_0}$ in the coordinate form
\begin{equation*}
\gamma^-_{\tilde x_0,\sigma_0}(t)=(x_{10}(t),x_{20}(t),x_{30}(t))
\end{equation*}
where $x_{10}$ is monotonely increases for $t\in [-T,0]$. Since continuous function can be approximated by smooth function, for small $\delta>0$, a tubular neighborhood of the semi-static curve $\gamma_{x_0,\sigma_0}^-|_{[-T,0]}$ admits smooth foliation of curves $\zeta_{x}$: $(x,t)\in (D+d_1)\times[-T,0]\to\mathbb{T}^3$ such that each semi-static curve $\gamma^-_{x,\sigma_0}|_{[-T,0]}$ remains $\delta$-close to $\zeta_{x}$ in the sense that $d(\zeta_x(t),\gamma_{x,\sigma_0}^-(t))<\delta$ for all $t\in[-T,0]$. Here we use $x$ to emphasize that the point is on the disk $D+d_1$.  The tubular neighborhood is defined by the form
$$
\text{\uj C}=\cup_{-T\le t\le 0}\{\zeta_{x}(t):x\in D+d_1\}.
$$

Let $\rho$: $(D+d_1)\times\mathbb{R}\to\mathbb{R}$ be a smooth function such that $\rho(x,t)=\rho(x',t)$, $\rho(x,t)=0$ if $t\notin[-T+t_0,-t_0]$ with small $t_0>0$ and $\rho(x,t)>0$ if $t\in (-T+t_0,-t_0)$. As $\zeta_x$ is a smooth foliation of the tubular domain, it can be thought as a differeomorphism $\Psi$: $(D+d_1)\times [-T,0]\to\text{\uj C}$, namely, for $\tilde x\in\text{\uj C}$ there exists unique $(x,t)\in (D+d_1)\times [-T,0]$ such that $\Psi(x,t)=\zeta_x(t)=\tilde x$. With a smooth function $V$: $D+d_1\to\mathbb{R}$ one obtains a smooth function $\tilde V$ defined on $\text{\uj C}$
\begin{equation}\label{completeeq5}
\tilde V(\tilde x)=\rho(\Psi^{-1}(\tilde x))V(\zeta_{x}(0)).
\end{equation}
By the construction of $\tilde V$, some constant $C_2>0$ exists such that
\begin{equation}\label{completeeq6}
\int_{-T+t_0}^{-t_0}\tilde V(\zeta_{x}(t))dt=C_2V(x), \qquad \forall\ x\in D+d_1.
\end{equation}

The potential perturbation is constructed in the form of (\ref{completeeq5}) where $V$ ranges over the function space spanned by
\begin{align*}
\mathfrak{V}_{2}=&\mu\Big(\sum_{\ell=1,2}a_{\ell}\cos2\ell\pi(x_2-x_{20}) +b_{\ell}\sin2\ell\pi(x_2-x_{20})\Big),\\
\mathfrak{V}_{3}=&\mu\Big(\sum_{\ell=1,2}c_{\ell}\cos2\ell\pi(x_3-x_{30}) +d_{\ell}\sin2\ell\pi(x_3-x_{30})\Big),
\end{align*}
each parameter of $(a_{\ell},b_{\ell},c_{\ell},d_{\ell})$ ranges over an unit interval $[1,2]$. If we construct a grid for the parameters $(a_{\ell},b_{\ell},c_{\ell},d_{\ell})$ by splitting the domain equally into a family of cubes and setting the size length by
$$
\Delta a_{\ell}=\Delta b_{\ell}=\Delta c_{\ell}=\Delta d_{\ell}=\mu,
$$
the grid consists of as many as $[\mu^{-8}]$ cubes.

Let us choose a neighborhood of the parameter $\sigma_0$, denoted by $\mathbb{I}_{\sigma_0}$ which satisfies the conditions:

1, for each $(x,\sigma)$ with $x\in D$, $\sigma\in\mathbb{I}_{\sigma_0}$, there is a unique backward semi-static curve $\gamma^-_{x,\sigma}$ such that $\gamma^-_{x,\sigma}(0)=x$ and $\gamma^-_{x,\sigma}(t)\to\tilde\Upsilon_{l,\sigma}$ as $t\to -\infty$. It is guaranteed by the existence of unstable manifold and if $D$ is suitably close to $\tilde\Upsilon_{l,\sigma_0}$. By the definition, $\gamma_{x,\sigma}^-(t)\in\text{\uj C}$ for $t\in [-T,0]$ and $x\in D$, so each $\sigma\in\mathbb{I}_{\sigma_0}$ defines a linear operator
\begin{equation}\label{completeeq7}
\mathscr{K}_{\sigma}\tilde V=\int_{-T}^0\tilde V(\gamma_{x,\sigma}^-(t))dt;
\end{equation}

2, as each curve $\gamma_{x,\sigma_0}^-(t)$ stays in $\delta$-neighborhood of the fiber $\zeta_{x}$ for $t\in [-T,0]$ with small $\delta>0$, by choosing suitably small neighborhood $\mathbb{I}_{\sigma_0}$ (depending on the size of $D$) some constant $C_3>0$ exists such that
\begin{align}\label{completeeq8}
\text{\rm Osc}_{x\in D}(\mathscr{K}_{\sigma}\tilde V-\mathscr{K}_{\sigma}\tilde V')&=\max_{x,x'\in D}|\mathscr{K}_{\sigma}\tilde V(x)-\mathscr{K}_{\sigma}\tilde V'(x')|\notag\\
&>\frac 12C_2\text{\rm Osc}_{x\in D}(V-V')\\
&>C_3\mu\Delta\notag
\end{align}
with $\Delta=\max\{|a_{\ell}-a'_{\ell}|,|b_{\ell}-b'_{\ell}|,|c_{\ell}-c'_{\ell}|, |d_{\ell}-d'_{\ell}|\}$. Indeed, as $V$ is a linear combination of the functions $\{\sin\ell x_j,\cos\ell x_j:\ell =1,2,j=2,3\}$, there exists some number $d=d(D)>0$ depending on the size of $D$ only such that the Hausdorff distance
$$
d_H\Big(V_{D}^{-1}\Big(\min_DV+\frac 14\Delta\Big),V_{D}^{-1}\Big(\max_DV-\frac 14\Delta\Big)\Big)\ge d(D)
$$
where $V_{D}^{-1}(\min_DV+\frac 14\Delta)=\{x\in D:V(x)\le\min _DV+\frac 14|(\max_DV-\min_DV)|\}$ and $ V_{D}^{-1}(\max_DV-\frac 14\Delta)=\{x\in D:V(x)\ge\max _DV-\frac 14|(\max_DV-\min_DV)|\}$. By requiring $\sigma$ suitably close to $\sigma_0$ and using the notation $\pi_x(x,t)=x$, we have
$$
\pi_x\Psi^{-1}\gamma_{x,\sigma}(t)\in V_{D}^{-1}(\min_DV+\frac 14\Delta)\qquad \text{\rm if}\ V(x)=\min_D V;
$$
and
$$
\pi_x\Psi^{-1}\gamma_{x,\sigma}(t)\in V_{D}^{-1}(\max_DV-\frac 14\Delta)\qquad \text{\rm if}\ V(x)=\max_D V.
$$
Therefore, one obtains (\ref{completeeq8}) from (\ref{completeeq5}), (\ref{completeeq6}) and (\ref{completeeq7});

3, for each $\sigma\in \mathbb{I}_{\sigma_0}$ and each $x\in D$, the forward semi-static curve $\gamma_{x,\sigma}^+$ with $\gamma_{x,\sigma}(0)=x\in D$ does not touch the support of $\rho\subset\text{\uj C}$ and approaches $\Upsilon_{r,\sigma}$ as $t$ increases to infinity.

Let $u^+_{r,\sigma}$ and $u^-_{\ell,\sigma}$ denote the weak KAM solution generating forward semi-static orbits approaching $\tilde\Upsilon_{r,\sigma}$ and backward semi-static orbits approaching $\tilde\Upsilon_{\ell,\sigma}$ respectively. Let $u^+_{r,\sigma,\tilde V}$ and $u^-_{l,\sigma,\tilde V}$ be the weak KAM solutions defined in the same way for the system under potential perturbation $L(\dot x,x)\to L(\dot x,x)+\tilde V(x)$ of \ref{completeeq5}.

Under such potential perturbation, the invariant cylinder remains unchanged. Restricted on the disk $D$, the function $u^+_{r,\sigma,\tilde V}$ is unchanged $(u^+_{r,\sigma,\tilde V}-u^+_{r,\sigma})|_{x\in D}=0$, but the function $u^-_{l,\sigma,\tilde V}$ undergoes small perturbation $u^-_{l,\sigma,\bar V}\neq u^-_{l,\sigma}$. To see how it is related to the potential, let us recall the following relations
$$
u^-_{l,\sigma}(\gamma_{x,\sigma}(0))-u^-_{l,\sigma}(\gamma_{x,\sigma}(-t))= \int_{-t}^0(L-\eta_{c})(d\gamma_{x,\sigma}(t))dt+Et
$$
if $\gamma_{x,\sigma}$ is a semi-static curve determined by $u^-_{l,\sigma}$ with $\gamma_{x,\sigma}(0)=x$ and $c\in I_{\sigma}$. We also have
$$
u^-_{l,\sigma,\tilde V}(\gamma_{x,\sigma}(0))-u^-_{l,\sigma,\tilde V}(\gamma_{x,\sigma}(-t))\le \int_{-t}^0(L+\tilde V-\eta_{c})(d\gamma_{x,\sigma}(t))dt+Et.
$$
Clearly, for suitably large $t$ the backward weak-KAM solution $\gamma_{x,\sigma}(-t)$ shall retreat into a small neighborhood of $\tilde\Upsilon_{l,\sigma}$ where the weak KAM solution $u^-_{l,\sigma}$ also remains unchanged. Therefore we deduce from the last two formulae that
$$
u^-_{l,\sigma,\tilde V}(x)-u^-_{l,\sigma}(x)\ge\int^0_{-T}\tilde V(\gamma_{x,\sigma}(t))dt.
$$
In a similar way, we find
$$
u^-_{l,\sigma,\tilde V}(x)-u^-_{l,\sigma}(x)\le\int^0_{-T}\tilde V(\gamma_{x,\sigma,\tilde V}(t))dt
$$
where $\gamma_{x,\sigma,\tilde V}$ stands for the backward semi-static curve determined by the weak-KAM solution $u^-_{l,\sigma,\tilde V}$ with $\gamma_{x,\sigma,\tilde V}(0)=x$. As $x$ lies in the region where $u^-_{l,\sigma,\tilde V}$ is differentiable, we have $|\gamma_{x,\sigma,\tilde V}(t)-\gamma_{x,\sigma}(t)| \to 0$ as $\tilde V\to 0$, guaranteed by the upper-semi continuity of semi-static curves. Therefore, it follows that for $x\in D$
\begin{align}\label{completeeq9}
u^-_{l,\sigma,\tilde V}(x)-u^-_{l,\sigma,\tilde V'}(x)=&\int_{-T}^0(\tilde V-\tilde V')(\gamma_{x,\sigma,\tilde V}^-(t))dt+o(\|\tilde V-\tilde V'\|),\\
=&(\mathscr{K}_{\sigma}+ \mathscr{R}_{\sigma})(\tilde V-\tilde V')\notag
\end{align}
where the linear operator $\mathscr{K}_{\sigma}$ is defined in (\ref{completeeq7}) and $\mathscr{R}_{\sigma}(\tilde V-\tilde V')=o(\|V-V'\|)$ in the sense of $C^0$-topology.

Next, let us consider all backward weak-KAM solutions for $\sigma\in\mathbb{I}_{\sigma}$.  Each parameter $\sigma\in\mathbb{I}_{\sigma}$ determines an interval $I_{c(\sigma)}$ for cohomology class. We restricted ourselves on a curve of first cohomology classes contained in the set $\cup I_{c(\sigma)}$ and intersecting each $I_{c(\sigma)}$ transversally. In this sense, we think the class defined on the interval $\mathbb{I}_c\ni c$ and the map $\sigma\to c(\sigma)$ is continuous. As  $h^{\infty}_{c(\sigma)}(\tilde x,\tilde x')=u^-_{l,\sigma}(\tilde x')-u^-_{l,\sigma}(\tilde x)$ if $\tilde x\in\tilde\Upsilon_{l,\sigma}$ and $h^{\infty}_{c(\sigma)}(\tilde x,\tilde x')=u^+_{r,\sigma}(\tilde x')-u^+_{r,\sigma}(\tilde x)$ if $\tilde x'\in\tilde\Upsilon_{r,\sigma}$, we obtain from Lemma 6.4 in \cite{CY2}
\begin{equation}\label{completeeq10}
\begin{aligned}
&|u^-_{l,\sigma}(x)-u^-_{l,\sigma'}(x)|\le
C_4(\sqrt{|\sigma-\sigma'|}+|c(\sigma)-c(\sigma')|),\\
&|u^+_{r,\sigma}(x)-u^+_{r,\sigma'}(x)|\le
C_4(\sqrt{|\sigma-\sigma'|}+|c(\sigma)-c(\sigma')|).
\end{aligned}
\end{equation}

We split the interval $\mathbb{I}_{\sigma}$ equally into $K_{\sigma}[\mu^{-2}]$ parts and split the interval $\mathbb{I}_c$ equally into $K_c[\mu^{-1}]$, where
$$
K_{\sigma}=\Big[L_{\sigma}\Big(\frac{12C_4}{C_3}\Big)^2\Big],\qquad
K_c=\Big[L_c\frac{12C_4}{C_3}\Big],
$$
$L_{\sigma}$ and $L_c$ are the length of $\mathbb{I}_{\sigma}$ and of $\mathbb{I}_c$ respectively. The grid over $\mathbb{I}_c\times\mathbb{I}_{\sigma}$ consists of as many as $K_{\sigma}K_c[\mu^{-3}]$ cuboids in which $K_{\sigma},K_c$ are independent of $\mu$. We pick up all cuboids which contain the points $(c,\sigma(c))$ and denote them by $\text{\uj c}_j$ with $j\in\mathbb{J}$, then the cardinality of the set $\mathbb{J}$ is not bigger than $K_{\sigma}K_c[\mu^{-3}]$.

According to the definition, a point $(c_j,\sigma(c_j))\in\text{\uj c}_j$ corresponds to a barrier function $u^-_{l,\sigma_j}-u^+_{r,\sigma_j}$. Let us assume that some parameters $(a_{\ell,j},b_{\ell,j})$ exist such that
$$
\text{\rm Osc}_{x\in D}\min_{x_3}\Big(u^-_{l,\sigma_j}-u^+_{r,\sigma_j}-(\mathscr{K}_{\sigma_j}+ \mathscr{R}_{\sigma_j})\tilde V_j\Big)=0
$$
where $\tilde V_j=\rho\Psi^{-1} V_j$ is defined as in (\ref{completeeq5}) with $V_j\in\mathfrak{V}_2$ determined by the parameters. We consider another perturbation determined by the parameters $(a'_{\ell},b'_{\ell})$
$$
V'=\mu\Big(\sum_{\ell=1,2}a'_{\ell}\cos2\ell \pi(x_2-x_{20})+b'_{\ell}\sin2\ell\pi(x_2-x_{20})\Big)
$$
and set $\tilde V'=\rho\Psi^{-1} V'$. By using the formula (\ref{completeeq9}) we write the identity
\begin{align*}
u^-_{l,\sigma,\tilde V'}-u^+_{r,\sigma,\tilde V'}&=(u^-_{l,\sigma,\tilde V'}-u^-_{l,\sigma_j,\tilde V'})- (u^+_{r,\sigma,\tilde V'}- u^+_{r,\sigma_j,\tilde V'})\\
&+(u^-_{l,\sigma_j}-u^+_{r,\sigma_j})-(\mathscr{K}_{\sigma_j}+\mathscr{R}_{\sigma_j})\tilde V_j\\
&+(\mathscr{K}_{\sigma_j}+\mathscr{R}_{\sigma_j})(\tilde V_j-\tilde V').
\end{align*}
For any point $(c,\sigma(c))\in\text{\uj c}_j$, in virtue of the formulae in (\ref{completeeq10}) the first term on the right-hand-side of the identity is not bigger than $C_3\mu^2/3$. For small $\|\tilde V_j-\tilde V'\|$ we have $\|(\mathscr{K}_{\sigma_j}+\mathscr{R}_{\sigma_j})(\tilde V_j-\tilde V')\|<\frac 13\|\mathscr{K}_{\sigma_j}(\tilde V_j-\tilde V')\|$. As both $V'$ and $V_j$ are independent of $x_3$, if the parameters $(a'_{\ell},b'_{\ell})$ satisfy
$$
\max\{|a_{\ell,j}-a'_{\ell}|,|b_{\ell,j}-b'_{\ell}|\}\ge\mu
$$
we find from above identities and the estimate (\ref{completeeq8}) that
\begin{equation}\label{completeeq11}
\text{\rm Osc}_{x\in D}\min_{x_3}\Big(u^-_{l,\sigma}-u^+_{u,\sigma}-(\mathscr{K}_{\sigma}+ \mathscr{R}_{\sigma})\tilde V'\Big)\ge \frac 13C_3\mu^2>0.
\end{equation}
It implies that, for each small rectangle $\text{\uj c}_j$ we only need to cancel out at most $2^4$ $\mu$-cubes from the grid for $\{\Delta a_{\ell},\Delta b_{\ell}:\ell=1,2\}$ so that the formula (\ref{completeeq11}) holds for the all other cubes. Let $j$ ranges over the set $\mathbb{J}$, we obtain a set $\text{\uj S}^c_2\subset\{a_{\ell}\in [1,2],b_{\ell}\in[1,2]:\ell=1,2\}$ with Lebesgue measure
$$
\text{\rm meas}\text{\uj S}^c_2\ge 1-2^4K_{\sigma}K_c\mu,
$$
such that the formula (\ref{completeeq11}) holds for each $(a'_{\ell},b'_{\ell})\in\text{\uj S}^c_2$ and for each $\sigma\in\mathbb{I}_{\sigma_0}$.

By considering $V'\in\mathfrak{V}_3$, in the same way we can see that some set $\text{\uj S}^c_3\subset\{c_{\ell}\in [1,2],d_{\ell}\in[1,2]:\ell=1,2\}$ with Lebesgue measure
$$
\text{\rm meas}\text{\uj S}^c_3\ge 1-2^4K_{\sigma}K_c\mu,
$$
such that the formula
\begin{equation}\label{completeeq12}
\text{\rm Osc}_{x\in D}\min_{x_2}\Big(u^-_{l,\sigma}-u^+_{u,\sigma}-(\mathscr{K}_{\sigma}+ \mathscr{R}_{\sigma})\tilde V'\Big)>0
\end{equation}
for each $(c'_{\ell},c'_{\ell})\in\text{\uj S}^c_3$ and each $\sigma\in\mathbb{I}_{\sigma_0}$.

Therefore, for each $(a_{\ell},b_{\ell},c_{\ell},d_{\ell},)\in\text{\uj S}^c_2\times\text{\uj S}^c_3$, the formulae (\ref{completeeq11}) and (\ref{completeeq12}) implies that for all $\sigma\in\mathbb{I}_{\sigma_0}$ the diameter of each connected component of the set
$$
\arg\min(u^-_{l,\sigma,\tilde V}-u^+_{r,\sigma,\tilde V})|_D
$$
is smaller than $D$. As $\mu>0$ can be arbitrarily small, for each disk $D$, an open-dense set $\mathfrak{V}_D$ exists such that this disconnect property holds for the system $L+\tilde V$ with $\tilde V\in\mathfrak{V}_D$. Since $\sigma$ is restricted on a closed set in the line which can be covered by finitely many $\mathbb{I}_{\sigma_i}$, this property is also open-sense for all $\sigma$ under our consideration.

Each section $D$ admits a hierachy of partition of small disks $D=\cup_jD_{kj}$ such that the size $D_{jk}$ approaches zero as $k\to\infty$, the intersection $\cap_k\mathfrak{V}_{D_{kj}}$ is a residual set. Therefore, we have proved
\begin{lem}\label{chainthm1}
There exists a set $\mathfrak{P}$ residual in $\mathfrak{P}_0$ such that for each $\epsilon P(q)\in\mathfrak{P}$ the set
$$
\arg\min(u^-_{l,\sigma}-u^+_{r,\sigma})\backslash((\Upsilon_{l,\sigma}\cup\Upsilon_{r,\sigma})+\delta)
$$
consists of totally disconnected semi-static curves.
\end{lem}

This lemma verifies the existence of generalized transition chain connecting any two class $\tilde c,\tilde c'\in\tilde{\mathbb{W}}_{k',\tilde E}\backslash\cup(\tilde{\mathbb{F}}_{i,\tilde E}+\epsilon^{\frac{1+d}2})$. As cohomology equivalence has been established along circles around the disk $\tilde{\mathbb{F}}_{i,\tilde E}$ which make up an annulus with width $O(\epsilon)\gg\epsilon^{\frac{1+d}2}$,  we obtains a generalized transition chain along the whole resonant channel $\tilde{\mathbb{W}}_{k',\tilde E}$. This completes the proof of Theorem \ref{mainresult2}.

\noindent{\bf Acknowledgement} The main content of this paper comes from the previous preprint \cite{C13} which seems quite long. The author splits it into several parts with modifications so that they are easier to read. This is one of them. The author is greatly appreciated for the discussion with Jinxin Xue.

This work is supported by National Basic Research Program of China (973 Program, 2013CB834100), NNSF of China (Grant 11171146) and a program PAPD of Jiangsu Province, China.

\end{document}